\documentclass [11pt,reqno]{amsart}
\usepackage {amsmath,amssymb,verbatim,geometry,color,esint}
\usepackage[all]{xy}
\usepackage[T1,T5]{fontenc}
\usepackage{mathrsfs}
\usepackage[backref,pagebackref,pdftex,hyperindex]{hyperref}
\usepackage[pdftex]{graphicx}
\usepackage{tikz}
\usetikzlibrary{math}

\usepackage{accents}

 \newcommand{\C}{\mathbb{C}}

\newcommand{\N}{\mathbb{N}}

 \newcommand{\Q}{\mathbb{Q}}
 \newcommand{\R}{\mathbb{R}}
 \newcommand{\Z}{\mathbb{Z}}

\newcommand{\cC}{\mathcal{C}}

\newcommand{\cI}{\mathcal{I}}

\newcommand{\cO}{\mathcal{O}}

\renewcommand{\a}{\alpha}
\renewcommand{\b}{\beta}
\renewcommand{\d}{\delta}
\newcommand{\e}{\varepsilon}
\newcommand{\f}{\varphi}

\newcommand{\g}{\gamma}
\newcommand{\la}{\lambda}
\newcommand{\om}{\omega}
\newcommand{\Om}{\Omega}

\newcommand{\p}{\psi}

\newcommand{\ie}{{\rm i.e.\ }}

\newcommand{\winter}{\wedge\dots\wedge}

\newcommand{\hto}{\hookrightarrow}

\newcommand{\ddc}{dd^c}

\newcommand{\BC}{\mathrm{BC}}

\DeclareMathOperator{\env}{P}

\newcommand{\lvol}{\underline{\mathrm{vol}}}
\newcommand{\uvol}{\overline{\mathrm{vol}}}

\DeclareMathOperator{\one}{\mathbf{1}}

\DeclareMathOperator{\supp}{supp}
\DeclareMathOperator{\vol}{vol}

\DeclareMathOperator{\PSH}{PSH}

\newcommand{\dbar}{\overline{\partial}}

\renewcommand{\div}{\mathrm{div}}

\newcommand{\D}{\Delta}

\newcommand{\HBC}{H_{\mathrm{BC}}}

\numberwithin{equation}{section}       

\newtheorem{prop} {Proposition} [section]
\newtheorem{thm}[prop] {Theorem} 
\newtheorem{defi}[prop] {Definition}
\newtheorem{lem}[prop] {Lemma}
\newtheorem{cor}[prop]{Corollary}
\newtheorem{prop-def}[prop]{Proposition-Definition}

\newtheorem*{thmA}{Theorem A} 
 
\newtheorem*{thmB}{Theorem B} 
\newtheorem*{thmC}{Theorem C} 
\newtheorem*{thmD}{Theorem D}

\newtheorem{exam}[prop]{Example}
\newtheorem{rmk}[prop]{Remark}

\theoremstyle{remark}
\newtheorem*{ackn}{Acknowledgment}

 \usepackage{hyperref}
\hypersetup{
    unicode=false,        
    pdftoolbar=true,      
    pdfmenubar=true,       
    pdffitwindow=false,     
    pdfstartview={FitH},    
    pdftitle={Volumes of Bott--Chern classes},    
    pdfauthor={BoucksomGuedjLu},     
    colorlinks=true,       
   linkcolor=black,          
    citecolor=black,        
    filecolor=black,      
    urlcolor=black}

\title[Volumes of Bott--Chern classes]{Volumes of Bott--Chern classes}

\date{\today}

\author{S{\'e}bastien Boucksom \and Vincent Guedj \and Chinh H. Lu}

\address{Sorbonne Universit\'e and Universit\'e Paris Cit\'e\\
CNRS\\
IMJ-PRG\\
F-75005 Paris\\
France}
\email{sebastien.boucksom@imj-prg.fr}

\address{Institut Universitaire de France \& Institut de Mathématiques de Toulouse\\
  Université de Toulouse; CNRS, UPS\\
  118 route de Narbonne, F-31400 Toulouse\\
  France}
  
\email{vincent.guedj@math.univ-toulouse.fr}

\address{Univ Angers, CNRS, LAREMA, SFR MATHSTIC\\
  2 Bd de Lavoisier, 49000 Angers\\
  France}
  
\email{hoangchinh.lu@univ-angers.fr}


\begin{document}

\begin{abstract} 
We study the volumes of transcendental and possibly non-closed Bott--Chern $(1,1)$-classes on
an arbitrary compact complex manifold $X$. We show that the latter belongs to the 
class $\cC$ of Fujiki if and only if it has {\it the bounded mass property}
-i.e. its Monge-Amp\`ere volumes have a uniform upper-bound- and there exists a closed Bott--Chern class with positive volume. 
This yields a positive answer to a conjecture of Demailly--P\u{a}un--Boucksom. 
To this end we extend to the hermitian context the notion of non-pluripolar products of currents, 
allowing for the latter to be merely  {\it quasi-closed} and {\it quasi-positive}.
We establish a  quasi-monotonicity property of Monge-Amp\`ere masses, and
moreover show the existence of solutions to degenerate complex Monge-Amp\`ere equations in
big classes, together with uniform a priori estimates. This extends to the hermitian context 
fundamental results of Boucksom--Eyssidieux--Guedj--Zeriahi.
\end{abstract}

\maketitle

\setcounter{tocdepth}{1}
\tableofcontents
%
%
%
%
\section*{Introduction}
%

Complex Monge-Amp\`ere equations are powerful tools in complex geometry, as demonstrated in foundational works ranging from Yau’s resolution of the Calabi conjecture \cite{Yau78} to the construction of singular canonical metrics by Eyssidieux-Guedj-Zeriahi \cite{EGZ09}. 
 As illustrated in \cite{BEGZ10}, the Bedford-Taylor pluripotential theory \cite{BT76,BT87,Kol98} is highly robust, extendable to the broad context of big cohomology classes. This extension has led to many profound applications in K\"ahler geometry over the past decade.
 
  Following Yau's celebrated work, the Hermitian Monge-Amp\`ere equation was explored by Cherrier \cite{Cher87}, who solved it in some particular cases. The general case solved by  Tosatti and Weinkove \cite{TW10} twenty years later underscores the difficulty of translating techniques from the K\"ahler setting to the Hermitian context. 
  In recent years, Hermitian pluripotential theory has become central to the advancement of hermitian geometry, as evidenced by 
  \cite{DK12, KN19, GL22,GL23}. This discussion underlines the necessity of extending the results of \cite{BEGZ10} to the Hermitian context, which is the primary objective of our paper.

\smallskip

We start by extending the non-pluripolar products of currents, allowing for the latter to be merely 
{\it quasi-closed} and {\it quasi-positive}.

\begin{thmA}
Let $(X,\omega_X)$ be a compact Hermitian manifold of dimension $n$  such that $\uvol(\omega_X)<+\infty$. 
For any $p=1,\dots,n$, there exists a unique symmetric, multilinear pairing 
$$
(T_1,\dots,T_p)\mapsto  T_1\winter T_p 
$$
defined on $p$-tuples of quasi-closed, quasi-positive currents and with values in nonpluripolar $(p,p)$-currents, such that:
\begin{itemize}
\item[(i)] if the $T_i$'s have bounded potentials, then the pairing is the Bedford-Taylor Monge-Amp\`ere product; 
\item[(ii)] for any tuple of $(1,1)$-forms $\theta_1,\dots,\theta_p$, the operator 
$$
(\f_1,\dots,\f_p)\mapsto \bigwedge_i(\theta_i+ \ddc\f_i), 
$$
defined on all tuples of quasi-psh functions, is local in the plurifine topology. 
\end{itemize}
Moreover, 
$
T_1,\dots,T_p\ge 0\Longrightarrow  T_1\winter T_p \geq 0. 
$
\end{thmA}

A $(1,1)$-current $T$ is quasi-closed if $dT$ is a smooth form.
It is quasi-positive if $T\geq -A \omega_X$, for some constant $A$. 
When $\theta=\theta_1=...=\theta_n$ and $u=u_1=...=u_n$, we simply write 
$$(\theta+dd^c u)^n=\bigwedge_i(\theta_i+ \ddc u_i).$$ 

Our proof is partially inspired by the work of Dinew-Ko{\l}odziej \cite{DK12}, which is further elaborated in \cite{GL22}. Due to the lack of local potentials of the forms $\theta_i$, we locally add strictly plurisubharmonic functions
to employ Bedford-Taylor’s construction of the Monge-Amp\`ere  measure. Specifically, 
we approximate the unbounded potentials through truncation and then consider the Bedford-Taylor product
\[
\bigwedge_i(\theta_i+ \ddc \max(\f_i,-t)).
\] 

However, this procedure results in currents of order zero with coefficients that are signed measures instead of positive ones, unless
$\theta_i\geq 0$. The key point is that when restricted to the plurifine open set $\cap_{i=1}^p(\f_i >-t)$ these signed measures are,
in fact, positive. 
To ensure our construction functions properly, we critically assume 
 that $X$ has the {\it bounded mass property}, meaning it satisfies
$$
\uvol(\omega_X):=\sup \left\{ \int_X (\omega_X+ \ddc \f)^n \; ; \; 
\f \in {\mathcal C}^{\infty}(X,\omega_X) 
\text{ with } \omega_X+ \ddc \f>0 \right\}<+\infty. 
$$

The condition $\uvol(\omega_X) <+\infty$ is independent of the positive form 
$\omega_X$, and  it is a bimeromorphic invariant \cite[Theorem A]{GL22}. In particular 
$\uvol(\omega_X) <+\infty$ if $X$ belongs to the class ${\mathcal C}$ of Fujiki.
Several conditions have been provided in \cite{AGL23} to ensure the finiteness of $\uvol(\omega_X)$. 
To date, we are not aware of any examples where $\uvol(\omega_X)$ is infinite.

\medskip

Our second main result establishes the following crucial quasi-monotonicity property of Monge-Amp\`ere masses,
which extends \cite{WN19,DDL2} to the hermitian setting.

\begin{thmB}
	Assume $\uvol(\omega_X)<+\infty$, and fix $\theta$  a smooth $(1,1)$-form 
	such that $\pm \theta \leq C \omega_X$.
If $\f,\p\in\PSH(\theta)$ satisfy $\f\le\p+O(1)$, then 
$$
\int_X (\theta+ \ddc\f)^n \le\int_X (\theta+ \ddc\p)^n+\D_\theta,
$$
where 
$
\D_\theta:=\sup_{\f,\p\in\PSH(C\om_X)\cap L^\infty(X)}\left\{\int_X(\theta+ \ddc\f)^n-\int_X(\theta+ \ddc\p)^n\right\}. 
$
\end{thmB}

Stokes' theorem shows that $\Delta_{\theta}=0$ if $\theta$ is closed.  In particular,  the total mass of $T^n$ is the same for all currents $T\in \{\theta\}$ with minimal singularities. It is thus natural to set 
\[
\vol(\{\theta\})=\int_X T_{\min}^n. 
\]
This, together with \cite{BEGZ10}, shows that our volume notion is compatible with the one introduced by Boucksom in \cite{B02}. 

The situation is more subtle when $\theta$ is not closed, as positive currents  with minimal singularities do not necessarily have the same Monge-Amp\`ere mass (unless $dd^c \theta=0$ and $dd^c \theta^2=0$, see \cite{GL22} and references therein).
When $\{\theta\}$ is {\it big} we thus introduce
$$
\lvol(\{\theta\}):=\lvol(T_{min})
\; \; \text{ and } \; \; 
\uvol(\{\theta\}):=\uvol(T_{min}),
$$
with
$$
\lvol(T):=\inf_S\int_X S^n,\quad \uvol(T):=\sup_S\int_X S^n, 
$$
where $S$ ranges over all quasi-closed positive $(1,1)$-currents in the same $\ddc$-class as $T$ and with equivalent singularities. 
By definition, the upper and lower volumes of $T$ only depend on its $\ddc$-class and singularity class, and satisfy 
$$
0\le\lvol(T)\le\uvol(T)<\infty, 
$$
as $X$ has the bounded mass property (see Lemma~\ref{lem:bdmass2}). Note also that 
\begin{equation*} 
T>0\Longrightarrow \uvol(T)>0. 
\end{equation*}
The analogous result for $\lvol(T)$ is unclear in general, even when $T$ is smooth (see~\S\ref{sec:posvol}).

\smallskip

We establish in Section \ref{sec:vol} key properties of these  volumes,
and notably show that $\lvol$ is a continuous function (Theorem \ref{thm:bigvol}).
Indeed another motivation of our paper stems from the  Grauert-Riemenschneider conjecture \cite{GR70},  which asks whether 
the existence of a semi-positive holomorphic line bundle $L \rightarrow X$  with $c_1(L)^n>0$
implies that $X$ is Moishezon, i.e. bimeromorphically equivalent to a projective manifold. 
This conjecture has been solved positively by Siu \cite{Siu84}  (see also \cite{Dem85}). 
One can  relax the semi-positivity assumption and replace
the condition $c_1(L)^n>0$ by the positivity of the volume $\vol(L)>0$. Recall that
$\vol(L)$ measures the asymptotic growth of the space of holomorphic sections,
$$
\vol(L)=\limsup_{k \rightarrow +\infty} \frac{n !}{k^n} h^0(X,kL).
$$
Demailly and P\u{a}un have further proposed a transcendental version of this conjecture
  \cite[Conjecture 0.8]{DP04}: given a nef class $\alpha \in H^{1,1}_{BC}(X,\R)$
      with $\alpha^n>0$, they conjectured that $\alpha$ should contain a K\"ahler current; 
       in particular $X$ should belong the  class ${\mathcal C}$.
      Recall that the Bott-Chern cohomology group $H_{BC}^{1,1}(X,\R)$
    is the quotient of   closed real smooth $(1,1)$-forms,
   by the image of  ${\mathcal C}^{\infty}(X,\R)$
   under the $\ddc$-operator.
   The following answer to \cite[Conjecture 0.8]{DP04} 
has been proposed in \cite[Theorem C]{GL22}:
$$
\alpha\text{ contains a K\"ahler current }
\Longleftrightarrow
\uvol(\omega_X) <+\infty.
$$    

The existence of a class $\alpha \in H^{1,1}_{BC}(X,\R)$ that is 
nef and satisfies $\alpha^n>0$ is however not granted on an arbitrary manifold in the Fujiki class
(see \cite{JM22}). It is thus natural to extend the conjecture
of Demailly-P\u{a}un, by removing the nef assumption,
as was proposed by Boucksom in \cite[Conjecture 4.1]{B02}.  

\smallskip

Theorem \ref{thm:bigvol} applied to closed forms provides a characterization of the Fujiki class,
and yields  the following  answer to the Boucksom-Demailly-P\u{a}un conjecture. 

\begin{thmC}
Let $(X,\omega_X)$ be a compact hermitian manifold.
The following  are equivalent:
\begin{enumerate}
\item $X$ belongs to the class ${\mathcal C}$ of Fujiki.
\item $\uvol(\omega_X)<+\infty$ and there exists a   class $\alpha \in H_{BC}^{1,1}(X,\R)$
such that $\vol(\alpha)>0$.
\end{enumerate}
\end{thmC}

In the final Section \ref{sec:solvMA} we extend some of the main results of \cite{BEGZ10}
by solving complex Monge-Amp\`ere equations associated to a (non-closed) big form $\theta$,
and by establishing uniform a priori estimates in this context, elaborating on \cite{GL23}.
We say that a form $\theta$ is {\it big} if there exists a $\theta$-psh function
$\rho$ with analytic singularities such that $\theta+dd^c \rho $ dominates a hermitian form.
We fix such a big form, and we no longer assume that $X$ has the bounded mass property.
We then extend the key result \cite[Theorem B]{BEGZ10} as follows.

\begin{thmD}
Assume $0 \leq f \in L^p(X)$ with $p>1$ and $\|f\|_p>0$.  Then there exists $(\varphi,c)\in \PSH(X,\theta)\times (0,+\infty)$ 
such that   $\sup_X \f=0$,
	$$
V_{\theta} -C \leq \varphi \leq V_{\theta} \; \; \text{ and } \; \; 	(\theta+dd^c \varphi)^n = c f \omega_X^n,
 $$
  where
	\begin{itemize}
	\item the constant $c>0$ is uniquely determined by $f,X,\theta, \omega_X$, and
	\item $C>0$ is a uniform constant that only depends on $(X,\omega_X)$, $\theta$, $p$ and $\|f\|_p$. 
	\end{itemize}  
\end{thmD}

Here $V_{\theta}=\sup \{ u, u \in PSH(\theta) \; \text{ with } \; u \leq 0 \}$ is a $\theta$-psh function with minimal singularities.
As in \cite[Section 6]{BEGZ10}, we also show the existence of solutions to Monge-Amp\`ere equations twisted by an exponential (Theorem \ref{thm:MAAY}). However instead of deducing Theorem \ref{thm:MAAY} from Theorem D by a fixed point argument, we first prove
Theorem \ref{thm:MAAY} and then deduce Theorem D by a deformation argument.

\smallskip

The uniqueness of the solution $\f$  is an open problem in the Hermitian setting, even when $\theta$ is a Hermitian form
(see \cite{KN19} for a positive result when $f$ is uniformly bounded away from zero).
Showing higher regularity of the solutions in the Zariski open set $\Omega=(\rho>-\infty)$,
 -when $f$ is smooth-
is an important open problem which is also largely open in the K\"ahler case
(\cite{BEGZ10} treats the case when the class $\{\theta\}$ is moreover nef).

\begin{ackn} 
The authors are partially supported by the fondation Charles Defforey  and the
Institut Universitaire de France. The third named author acknowledges partial support from the PARAPLUI ANR-20-CE40-0019 project and the Centre Henri Lebesgue ANR-11-LABX-0020-01.
\end{ackn}

%

\section{Quasi-closed and quasi-positive currents}

Throughout the article, $X$ denotes a compact complex manifold, of complex dimension $n$. We also pick a reference Hermitian $(1,1)$-form $\om_X>0$. The purpose of this section is to review some basic properties of quasi-psh functions, quasi-closed quasi-positive $(1,1)$-currents, and their $dd^c$-classes.

%
%
\subsection{Quasi-closed currents and $\ddc$-classes}
For each $p,q\in\N$, denote by $\Om^{p,q}(X)$ the Fr\'echet space of smooth $(p,q)$-forms on $X$. 

\begin{defi} We define the \emph{Bott--Chern space} $\BC^{p,q}(X)$ of (possibly non-closed) \emph{$\ddc$-classes of bidegree $(p,q)$} as the cokernel of $\ddc\colon\Om^{p-1,q-1}(X)\to\Om^{p,q}(X)$, \ie 
$$
\BC^{p,q}(X):=\frac{\Om^{p,q}(X)}{\ddc\Om^{p-1,q-1}(X)}. 
$$
\end{defi}
We denote by $\{\theta\}\in\BC^{p,q}(X)$ the $\ddc$-class of a $(p,q)$-form $\theta$. When $p=q$, the real operator $\ddc$ induces a map $\ddc\colon\Om^{p-1,p-1}(X,\R)\to\Om^{p,p}(X,\R)$ between spaces of real forms, and we denote by 
$$
\BC^{p,p}(X,\R)\subset\BC^{p,p}(X)
$$
its cokernel. Setting $\bar d\{\theta\}:=d\theta$ for any $(p,q)$-form $\theta$ defines a linear map
\begin{equation}\label{equ:bard}
\bar d\colon\BC^{p,q}(X)\to\Om^{p+1,q}(X)\oplus\Om^{p,q+1}(X),
\end{equation}
whose kernel coincides with the usual \emph{Bott--Chern cohomology space} 
$$
\HBC^{p,q}(X)=\frac{\Om^{p,q}(X)\cap\ker d}{\ddc\Om^{p-1,q-1}(X)}\hto\BC^{p,q}(X). 
$$
When $X$ is K\"ahler (or even Fujiki), the latter coincides with the Dolbeault space $H^{p,q}(X)$. 
 
We endow the (infinite dimensional) complex vector space $\BC^{p,q}(X)$ with the quotient topology. Thus $\{\theta_j\}\to\{\theta\}$ in $\BC^{p,q}(X)$ iff $\theta_j\to\theta$ smoothly for an appropriate choice of representatives. Note that~\eqref{equ:bard} is continuous in this topology. 

\begin{lem}\label{lem:Frechet} The space $\HBC^{p,q}(X)$ is finite dimensional, and $\BC^{p,q}(X)$ is a (Hausdorff) Fr\'echet space.
\end{lem}

\begin{proof} The first point is well-known, see for instance~\cite[Theorem~VI.12.4]{DemBook}. It equivalently says that $\ddc\Om^{p-1,q-1}(X)$ has finite codimension in the closed subspace $\Om^{p,q}(X)\cap\ker d$ of the Fr\'echet space $\Om^{p,q}(X)$. Thus $\ddc\Om^{p-1,q-1}(X)$ is closed in $\Om^{p,q}(X)\cap\ker d$, and hence in $\Om^{p,q}(X)$, which proves the second point. 
\end{proof}

As is well-known, Bott--Chern cohomology can also be described in terms of currents. This is more generally the case 
for $\BC^{p,q}(X)$. To this end, it will be convenient to introduce the following terminology. 

\begin{defi} We say that a current $T$ on $X$ is \emph{quasi-closed} if $dT$ is smooth.
\end{defi}
In particular, any smooth form is quasi-closed. 
\begin{lem}\label{lem:quasi} A $(p,q)$-current $T$ on $X$ is quasi-closed iff it admits a decomposition 
$$
T=\theta+\ddc U
$$
where $\theta$ is a smooth $(p,q)$-form and $U$ a $(p-1,q-1)$-current. Furthermore, any other such decomposition is of the form 
$$
T=(\theta+\ddc\g)+\ddc(U-\g)
$$ 
for a smooth $(p-1,q-1)$-form $\g$. 
\end{lem}

\begin{proof} 
While this result is likely well-known, we provide the simple argument for the convenience of the reader. The existence of a current $U$ such that $T-\ddc U$ is smooth implies that $dT=d(T- \ddc U))$ is smooth. 
Conversely, assume $dT$ is smooth. We first claim that there exists a $(p+q-1)$-current $S$  such that $T+dS$ is smooth. Indeed, since de Rham cohomology can be computed using either forms or currents, the smooth $(p+q+1)$-form $dT$ is exact, and we can thus find a $(p+q)$-form $\a$ such that $dT=d\a$. The $(p+q)$-current $T-\a$ is then closed, and we can now find a $(p+q-1)$-current $S$ such that $T-\a+dS$ is smooth, hence the claim. 

Since $T+dS=T+(\partial+\dbar)S$ is smooth, decomposing according to type shows that
$$
T+\partial S^{p-1,q}+\dbar S^{p,q-1},\quad\dbar S^{p-1,q}\quad\text{and}\quad\partial S^{p,q-1}
$$
are smooth. Arguing as above, this time with Dolbeault cohomology, it follows that
$$
S^{p-1,q}+ {\bar{\partial} R}\quad\text{and}\quad S^{p,q-1}+\partial Q
$$
are smooth for some $(p-1,q-1)$-currents $R,Q$. Thus 
$$
T+\partial S^{p-1,q}+\dbar S^{p,q-1}=T+\partial\dbar (R-Q)=:\theta, 
$$
is smooth, which yields the desired decomposition $T=\theta+\ddc U$. 

Consider now another decomposition $T=\theta'+\ddc U'$. Then $\ddc(U-U')$ is a smooth $(p,q)$-form, which $\ddc$-exact as a current, and hence as a form. Thus $\ddc(U-U')=\ddc\g$ for a smooth $(p-1,q-1)$-form $\g$, and the last point follows. 
\end{proof}
As a direct consequence of Lemma~\ref{lem:quasi}, the inclusion of smooth forms into quasi-closed currents induces a vector space isomorphism 
\begin{equation}\label{equ:BCcurrent}
\BC^{p,q}(X)\simeq\frac{\{T\text{ quasi-closed }(p,q)\text{-current}\}}{\{ \ddc U\mid U\,(p-1,q-1)\text{-current}\}}.
\end{equation}
We denote by $\{T\}\in \BC^{p,q}(X)$ the $\ddc$-class of a quasi-closed $(p,q)$-current $T$. Note that
$$
\bar d\{T\}=dT. 
$$
When $T$ is closed, $\{T\}$ lies in the finite dimensional space $H^{p,q}_\BC(X)$, and $T\mapsto\{T\}$ is further weakly continuous on closed currents, since the weak quotient topology of $H^{p,q}_\BC(X)$ is Hausdorff, and hence coincides with the strong quotient topology. 

However, $T\mapsto\{T\}$ is \emph{not} weakly continuous on quasi-closed currents. Indeed, a weakly convergent sequence 
of quasi-closed currents $T_j\to T$ can only satisfy $\{T_j\}\to\{T\}$ if $\bar d\{T_j\}=dT_j$ converges smoothly to $\bar d\{T\}=dT$.

%
\subsection{Quasi-psh functions}\label{sec:quasipsh}
Recall that a \emph{quasi-plurisubharmonic} function (\emph{quasi-psh} for short) is a function $\f\colon X\to\R\cup\{-\infty\}$ that is locally the sum of a psh function and a smooth function. In particular, $\f$ is usc and integrable. Quasi-psh functions are actually in $L^p(X)$ for any $p\in [1,\infty)$, and the induced topologies are further all equivalent. 

The \emph{plurifine topology} of $X$ is defined as the topology generated by all quasi-psh functions on $X$. We shall say that an operator $(\f_1,\dots,\f_p)\mapsto F(\f_1,\dots,\f_p)$, defined on certain tuples of quasi-psh functions and with values in order $0$ currents, is \emph{local in the plurifine topology} if, for any plurifine open $O\subset X$ and functions $\f_i,\p_i$, we have 
$$
\f_i=\p_i\text{ on }O\text{ for all }i\Longrightarrow\one_OF(\f_1,\dots,\f_p)=\one_O F(\p_1,\dots,\p_p).
$$

Any quasi-psh function $\f$ satisfies $\theta+\ddc\f\ge 0$ in the sense of currents for some smooth $(1,1)$-form $\theta$.  We then say that $\f$ is \emph{$\theta$-psh}, and denote by 
$$
\PSH(\theta)=\PSH(X,\theta)
$$
the set of such functions, which sits as a closed convex subspace of $L^1(X)$. 

When $\theta$ is closed, it admits a local potential near each point of $X$, \ie a smooth function $\rho$ such that $\theta= \ddc\rho$, and a function $\f$ is then $\theta$-psh iff $\rho+\f$ is psh. The following well-known observation allows to reduce many of the basic properties of $\theta$-psh functions for a possibly non-closed $\theta$ to the closed case. 

\begin{lem}\label{lem:approxclosed} Pick a real $(1,1)$-form $\theta$, and $\e>0$. Then any point of $X$ admits a neighborhood $U\subset X$ with a closed $(1,1)$-form $\theta'$ such that $-\e\om_X\le\theta-\theta'\le\e\om_X$ on $U$. 
\end{lem}

\begin{proof} 
For any $x\in X$, it suffices to pick local coordinates $(z_1,\dots,z_n)$  centered at $x$, and to define $\theta'$ as the constant $(1,1)$-form in these coordinates which coincides with $\theta$ at $x$. The result then holds by continuity of $\theta$ at $x$.
\end{proof}

Given any $(1,1)$-form $\theta$ and any function $f\colon X\to [-\infty,\infty]$, the \emph{$\theta$-psh envelope} of $f$ is defined as the pointwise supremum
\begin{equation}\label{equ:pshenv}
\env_\theta(f):=\sup\left\{\f\mid\f\in\PSH(\theta),\,\f\le f\right\}.
\end{equation}
Denote by $\env_\theta^\star(f)$ its usc regularization. Then one of the following holds: 
\begin{itemize}
\item[(a)] $\env_\theta^\star(f)\equiv-\infty$ (\ie there exists no $\theta$-psh function $\f$ such that $\f\le f$); 
\item[(b)] $\env_\theta^\star(f)\equiv+\infty$; 
\item[(c)] $\env_\theta^\star(f)$ is the largest $\theta$-psh function such that $\env_\theta^\star(f)\le f$ outside a pluripolar subset of $X$. 
\end{itemize}
Note that if $f$ is usc and (c) holds then $\env_\theta(\f)$ is already usc (and hence $\theta$-psh), since $\env^\star_\theta(\f)$ is then a candidate in the envelope. 

\begin{defi}\label{defi:analsing} We say that a quasi-psh function $\f$ has \emph{analytic singularities (with smooth remainder)} if $\f$ can be locally written as 
$$
\f=\frac{1}{2m}\log\sum_{j=1}^N|f_j|^2+u
$$
for some $m\in\Z_{>0}$ and some choice of holomorphic functions $f_1,\dots,f_N$ and smooth function $u$. If we can further take $N=1$ then we say that $\f$ has \emph{divisorial singularities (with smooth remainder)}.
\end{defi}
If $\f$ has analytic singularities, then it is smooth on the Zariski open set $\{\f>-\infty\}$. If $\f$ further has divisorial singularities, then the effective $\Q$-divisor $E:=\tfrac 1m\div(f_1)$ is globally defined, and the Lelong--Poincar\'e formula yields
$$
dd^c\f=\b+[E]
$$
where $\b$ is a closed smooth $(1,1)$-form. 

As pointed out in~\cite[Remark 2.7]{D+}, the notion of analytic singularities with smooth remainder needs to be handled with care, as it depends on the choice of local holomorphic functions $(f_j)$, and not only on the ideal sheaf they generate. At any rate, the following global result holds: 

\begin{lem}\label{lem:analres} Assume $\f$ is a quasi-psh function with analytic singularities. Then there exists a modification $\pi\colon Y\to X$, isomorphic over the Zariski open set $\{\f>-\infty\}$, such that $\pi^\star\f$ has divisorial singularities. 
\end{lem}
\begin{proof} By assumption, $X$ admits a finite open cover $(U_\a)$ such that 
$$
\f|_{U_\a}=\frac{1}{2m_\a}\log\sum_{j=1}^{N_\a}|f_{\a j}|^2+u_\a
$$
with $f_{\a j}\in\cO(U_\a)$, $u_\a\in C^\infty(U_\a)$ and $m_\a\in\Z_{>0}$. Set $m:=\prod_\a m_\a$, and denote by $\cI\subset\cO_X$ the ideal sheaf of germs of holomorphic functions $f$ such that $|f|\le C e^{m\f}$ locally. On $U_\a$ this condition is equivalent to $|f|\le C\max_j \left|f_{\a j}^{m/m_\a}\right|$. By the well-known analytic characterization of integral closure, this shows that $\cI|_{U_\a}$ coincides with the integral closure of the ideal sheaf $\cI_\a\subset\cO_{U_\a}$ generated by $(f_{\a j}^{m/m_\a})_{1\le j\le N_\a}$, and it follows that the ideal sheaf $\cI$ is coherent. 

Pick a log resolution $\pi\colon Y\to X$ of $\cI$, and denote by $D$ the effective divisor of $Y$ such that $\cI\cdot\cO_Y=\cO_Y(-D)$. Then $\pi$ factors through the normalization of the blowup of $\cI$, which coincides on $U_\a$ with the normalization of the blowup of $\cI_\a$, as the two ideal sheaves share the same integral closure. As a result, the functions $(\pi^\star f_{\a j}^{m/m_\a})_j$ generate $\cO(-D)$ on $\pi^{-1}(U_\a)$. Given a local equation $f_D$ of $D$, we thus have $\pi^\star f_{\a j}^{m/m_\a}=f_D g_{\a j}$ for a family of local holomorphic functions $(g_{\a j})_j$ without common zeroes. This shows 
$$
\pi^\star\f=\frac{1}{2m_\a}\log\sum_j |f_D|^{2m_\a/m} |g_{\a j}|^{2m_\a/m}+\pi^\star u_\a=\frac{1}{2m}\log|f_D|+v_\a
$$
with $v_\a:=u_\a+\frac{1}{2m_\a}\log\sum_j |g_{\a j}|^{2m_\a/m}\in C^\infty$, which proves, as desired, that $\pi^\star\f$ has divisorial singularities.
\end{proof}

The importance of functions with analytic singularities stems from the following consequence of Demailly's fundamental regularization theorem \cite{Dem92}: 

\begin{thm}[Demailly] \label{thm:Demreg} 
Any $\f\in\PSH(\theta)$ can be written as the limit of a decreasing sequence $\f_j\in\PSH(\theta+\e_j\omega_X)$ with analytic singularities, where $\e_j>0$ decreasing to zero. 
\end{thm}

%
\subsection{Quasi-positive currents}\label{sec:qclosed}
A $(1,1)$-current $T$ is said to be \emph{quasi-positive} if $T\ge\g$ for some smooth $(1,1)$-form $\g$. As a consequence of Lemma~\ref{lem:quasi}, we then have:

\begin{lem}\label{lem:qquasi} For any $(1,1)$-current $T$, the following are equivalent: 
\begin{itemize}
\item[(i)] $T$ is quasi-closed and quasi-positive;
\item[(ii)] $T=\theta+\ddc\f$ for a smooth real $(1,1)$-form $\theta$ and a quasi-psh function $\f$.
\end{itemize}
\end{lem}
Furthermore, any other decomposition as in (ii) is of the form 
$$
T=(\theta+\ddc\tau)+\ddc(\f-\tau)
$$
with $\tau\in C^\infty(X)$. As a result, the pull-back of a quasi-closed, quasi-positive $(1,1)$-current $T$ by any surjective holomorphic map $f\colon Y\to X$ can be defined by setting
$$
f^\star T:=f^\star\theta+\ddc f^\star\f
$$
for any decomposition $T=\theta+\ddc\f$ as above. In constrast, quasi-closed, quasi-positive currents are in general not preserved under push-forward.

\begin{exam} Assume $n=2$, and let $\pi\colon Y\to X$ be the blowup of a point $p \in X$. Then there exists a Hermitian form $\om_Y$ on $Y$ such that $\pi_\star\om_Y$ is not quasi-closed. Indeed, since any smooth $(1,1)$-form can be written as a difference of Hermitian forms, it would otherwise follow that $d\pi_\star\theta=\pi_\star d\theta$ is smooth for any smooth $(1,1)$-form $\theta$. 
However this fails for 
$$
\theta=\chi  dd^c (\pi^* f),
$$ 
where $\chi\in C^\infty_c(Y)$ is a cut-off function supported near a point of the exceptional divisor, 
and $f\in C^\infty(X)$ coincides with $\|z\|^2$ in a local chart centered at $p$. 
Indeed  
$$
\pi_\star d\theta = d\chi \circ \pi^{-1} \wedge dd^c \|z\|^2
$$
is a differential form in $X \setminus \{p\}$ whose coefficients are not even bounded near $p$.
\end{exam}

Extending standard terminology for closed positive currents, we shall say that a quasi-closed, quasi-positive current $T=\theta+dd^c\f$

\begin{itemize}

\item has \emph{bounded potentials} on a given compact subset $K\subset X$ if $\f$ is bounded on $K$;

\item is \emph{more singular than} $T'=\theta'+dd^c\f'$ if $\f\le\f'+O(1)$, and that $T,T'$ have \emph{equivalent singularities} if $\f=\f'+O(1)$; 

\item has \emph{analytic singularities} if the quasi-psh function $\f$ has analytic singularities (see Definition~\ref{defi:analsing}). 

\end{itemize} 
As consequence of Lemma~\ref{lem:analres}, we have:

\begin{lem}\label{lem:resolv} For any quasi-closed, quasi-positive current $T$ with analytic singularities, there exists a modification $\pi\colon Y\to X$ such that 
$$
\pi^\star T=\beta+[E]
$$ 
where $\b$ is a smooth $(1,1)$-form and $E$ an effective $\Q$-divisor. 
\end{lem}
Note further that $T$ is closed iff $\b$ is. 
%
%
%
%

%
\subsection{Positivity for $\ddc$-classes}\label{sec:posclass}
Recall that the space
$$
\BC^{1,1}(X,\R)=\Om^{1,1}(X,\R)/\ddc C^\infty(X)
$$
of real, bidegree $(1,1)$ $\ddc$-classes is a Fr\'echet space (see Lemma~\ref{lem:Frechet}), containing the usual (finite dimensional) Bott--Chern cohomology space $H^{1,1}_\BC(X,\R)$. 

\begin{defi} We say that a class in $\BC^{1,1}(X,\R)$ is:
\begin{itemize}
\item \emph{Hermitian} if it can be represented by a Hermitian form; 
\item \emph{nef} if it is a limit of Hermitian classes.
\end{itemize}
\end{defi}
The set of Hermitian classes is an open convex cone in $\BC^{1,1}(X,\R)$, which we call the \emph{Hermitian cone}. Its closure, the \emph{nef cone}, is the set of all nef classes.  

\begin{lem} A class $\{\theta\}\in\BC^{1,1}(X,\R)$ is nef iff it admits a sequence of representatives $\theta_j$ such that $\theta_j\ge-\e_j\om_X$ with $\e_j\to 0$. 
\end{lem}
\begin{proof} Assume the existence of representatives $\theta_j=\theta+\ddc\f_j$ such that $\theta_j\ge-\e_j\om_X$. Then $\theta'_j:=\theta_j+2\e_j\om_X>0$, and $\theta'_j-\ddc\f_j=\theta+2\e_j\om_X\to\theta$ smoothly, which shows that $\{\theta\}=\lim_j\{\theta_j\}$ is nef. Conversely, assume $\{\theta\}$ is nef. Then $\{\theta\}=\lim_j\{\om_j\}$ for a sequence of Hermitian forms $\om_j$, and hence $\om_j+dd^c\tau_j\to\theta$ for some smooth functions $\tau_j$. In particular, $\theta-(\om_j+\ddc\tau_j)\ge-\e_j\om_X$ with $\e_j\to 0$, and hence $\theta-\ddc\tau_j\ge-\e_j\om_X$. 
\end{proof}

\begin{defi} We say that a class $\{\theta\}\in\BC^{1,1}(X,\R)$ is:
\begin{itemize}
\item \emph{pseudo-effective} (\emph{psef} for short) if it can be represented by a positive $(1,1)$-current $T=\theta+\ddc\f\ge 0$; 
\item \emph{big} if it can be represented by a current $T=\theta+\ddc\f$ which is \emph{strictly positive}, \ie $T\ge\e\om_X$ for $0<\e\ll 1$. 
\end{itemize}
\end{defi}

As a consequence of Demailly's regularization  (Theorem~\ref{thm:Demreg} above), any big class can be represented by a strictly positive current with analytic singularities. 

\begin{lem}\label{lem:psefbig} The psef cone of $\BC^{1,1}(X,\R)$ is a closed convex cone, whose interior coincides with the big cone.
\end{lem}

In particular, the Hermitian cone is contained in the big cone, and the nef cone is contained in the psef cone. 

\begin{proof} It is straightforward to check that the big cone coincides with the interior of the psef cone. To see that the latter is closed, assume $\{\theta\}\in\BC^{1,1}(X,\R)$ is the limit of a sequence of psef classes $\{\theta_j\}$. For an appropriate choice of representatives, we then have $\theta_j\to\theta$ smoothly, and hence $\theta_j\le C\om_X$ for a uniform constant $C>0$. Since $\{\theta_j\}$ is psef, we can shoose $\f_j\in\PSH(\theta_j)\subset\PSH(C\om_X)$, normalized by $\sup_X\f_j=0$. By weak compactness of normalized $C\om_X$-psh functions, we may assume, after passing to a subsequence, that $\f_j$ converges weakly to a $C\om_X$-psh function $\f$. Since $\theta_j+\ddc\f_j\ge 0$, we infer $\theta+\ddc\f\ge 0$, which shows that $\{\theta\}$ is psef. 
\end{proof}

\begin{rmk} When $X$ is K\"ahler, the big cone of $H^{1,1}_\BC(X,\R)$, \ie its intersection with the big cone of $\BC^{1,1}(X,\R)$, coincides with the interior of the psef cone of $H^{1,1}_\BC(X,\R)$. However, this fails in general in the non-K\"ahler case: if $X$ is for instance a Hopf surface, then $H^{1,1}_\BC(X,\R)$ is one-dimensional, and its psef cone has nonempty interior since $X$ carries a nonzero closed positive $(1,1)$-current; however, $X$ does not admit
any closed strictly positive current, and the big cone of $H^{1,1}_\BC(X,\R)$ is thus empty. 
\end{rmk}

For later use, we also record the following useful characterization of psef classes, which follows from the Hahn--Banach theorem (see~\cite[Lemma 3.3]{Lam99}): 

\begin{lem}\label{lem:HB} A class $\{\theta\}\in\BC^{1,1}(X,\R)$ is psef iff $\int_X\theta\wedge\g\ge 0$ for every $\ddc$-closed, positive $(n-1,n-1)$-form $\g>0$. 
\end{lem}

A quasi-closed positive current $T$ has \emph{minimal singularities} (within its $\ddc$-class) if it is less singular than any other positive current in the psef class $\{T\}\in\BC^{1,1}(X,\R)$. As observed by Demailly, such currents exist in any psef class $\{\theta\}$. Indeed, the upper envelope 
$$
V_\theta:=\env_\theta(0)=\sup\left\{ \f\in\PSH(\theta)\mid\f\le 0\text{ on } X \right \}
$$ 
yields a positive current 
$T=\theta+ \ddc V_\theta$ with minimal singularities. Currents with minimal singularities in a psef class are however usually far from unique.

%
\subsection{The bounded mass property}\label{sec:bdmass}

Following~\cite{GL22} we introduce the following:

\begin{defi}
We define the \emph{upper volume} of $\om_X$ as 
$$
\uvol(\om_X):=\sup_\f\int_X(\om_X+ \ddc\f)^n\in [0,+\infty], 
$$
where $\f$ ranges over all smooth $\om_X$-psh functions. 
\end{defi}

The above supremum is unchanged if $\f$ ranges instead over all bounded $\om_X$-psh functions, since any such $\f$ is  a decreasing limit of smooth $\om_X$-psh functions $\f_j$ (by Theorem~\ref{thm:Demreg}), and 
$$
\int_X(\om+ \ddc\f_j)^n\to\int_X(\om+ \ddc\f)^n
$$
by continuity of Bedford--Taylor products along decreasing sequences. 

Any other Hermitian form $\om'_X$ satisfies $C^{-1}\om_X\le\om'_X\le C\om_X$ for some $C>0$, which implies $C^{-1}\PSH(\om_X)\subset\PSH(\om'_X)\subset C\PSH(\om_X)$, and hence 
$$
C^{-n}\uvol(\om_X)\le\uvol(\om'_X)\le C^n\uvol(\om_X).
$$
We may thus introduce: 

\begin{defi} Let $X$ be compact complex manifold. We say that $X$ has the \emph{bounded mass property} if 
$\uvol(\om_X)$ is finite for some (hence any) Hermitian form $\om_X$. 
\end{defi}

The following properties are straightforward:
\begin{itemize}
\item the bounded mass property holds whenever $\ddc\om_X^p=0$ for all $p$ (which, as is well-known, follows from the case $p=1,2$);  
\item it is inherited by any holomorphic image of $X$. 
\end{itemize}
In particular, the bounded mass property holds whenever $n\le 2$ (using a Gauduchon metric $\om_X$), and also when $X$ is a Fujiki manifold (\ie bimeromorphic to a compact K\"ahler manifold). 

By~\cite[Theorem~A]{AGL23} and~\cite[Theorem~3.7]{GL22}, we further have:

\begin{prop} \label{pro:bmp}
The bounded mass property is inherited by complex submanifolds, and is bimeromorphically invariant. 
\end{prop}

In dimension $n=3$, the bounded mass property is known to hold in the following cases (see~\cite{AGL23}):
\begin{itemize}
\item  there exists a Hermitian form $\om_X$ with $\ddc\om_X\ge 0$;
\item for `most' complex nilmanifolds; 
\item for the twistor space of a compact antiselfdual $4$-manifold.
\end{itemize}
At the time of this writing, we do not know any example where the bounded mass property fails, and it is an intriguing problem whether it always holds. For instance: 

\begin{exam} Does the Hopf manifold $X=(\C^n\setminus\{0\})/2^\Z$ of dimension $n\ge 3$ have the bounded mass property? See~\cite[\S 3.4.2]{AGL23} for a discussion. 
\end{exam}

%
\section{Non-pluripolar products} \label{sec:nonplpr}

As before, $X$ denotes a compact complex manifold, of complex dimension $n$, endowed with a reference Hermitian form $\om_X$. We show that the bounded mass property yields a notion of non-pluripolar products of quasi-closed, quasi-positive currents, and extend to this context some basic properties of non-pluripolar masses. 
%
\subsection{Products of currents with bounded potentials} \label{sect:bdd}
For any $p=1,\dots,n$, the classical work of Bedford--Taylor~\cite{BT87} extends the operator
$$
(\f_1,\dots,\f_p)\mapsto  \ddc\f_1\winter \ddc\f_p
$$
from tuples of (locally defined) smooth psh functions to bounded ones, with values in closed positive $(p,p)$-currents that are \emph{nonpluripolar}, in the sense that they put no mass on pluripolar sets. This operator is symmetric and multilinear, continuous along decreasing sequences, and further has the key property of being local in the plurifine topology (see~\S\ref{sec:quasipsh}). 

Since any quasi-psh function is locally the sum of a psh function and a smooth function, these results readily extend to tuples of bounded quasi-psh functions $\f_1,\dots,\f_p$, the closed $(p,p)$-current $\ddc\f_1\winter \ddc\f_p$ being now merely of order $0$ (\ie with signed measures as local coefficients), and still nonpluripolar. 

More generally, given arbitrary $(1,1)$-forms $\theta_1,\dots,\theta_p$, we define an operator
\begin{equation}\label{equ:MAbd}
(\f_1,\dots,\f_p)\mapsto\bigwedge_i(\theta_i+ \ddc\f_i)
\end{equation}
on tuples of bounded quasi-psh functions by forcing multilinearity, \ie 
\begin{equation}\label{equ:prodbd}
\bigwedge_{i=1}^p(\theta_i+ \ddc\f_i):=\sum_{I\subset\{1,\dots,p\}}\bigwedge_{i\in I}\theta_i\wedge\bigwedge_{i\notin I} \ddc\f_i. 
\end{equation}
This formula shows that the operator~\eqref{equ:MAbd} is still local in the plurifine topology. As the notation suggests, one easily checks that~\eqref{equ:prodbd} only depends on the currents $T_i=\theta_i+ \ddc\f_i$, and yields a symmetric, multilinear pairing 
\begin{equation}\label{equ:wedgebd}
(T_1,\dots,T_p)\mapsto T_1\winter T_p
\end{equation}
from tuples of quasi-closed, quasi-positive currents with bounded potentials to nonpluripolar $(p,p)$-currents. We next extend this as follows:

\begin{prop}\label{prop:prodbd} 
For any compact subset $K\subset X$, there exists a unique symmetric, multilinear pairing 
$$
(T_1,\dots,T_p)\mapsto\one_K T_1\winter T_p
$$
defined on tuples of quasi-closed, quasi-positive currents with bounded potentials on $K$ and with values in nonpluripolar $(p,p)$-currents, such that:
\begin{itemize}
\item[(i)] when the $T_i$'s have bounded potentials, the pairing is compatible with~\eqref{equ:wedgebd}; 
\item[(ii)] for any tuple of $(1,1)$-forms $\theta_1,\dots,\theta_p$, the operator 
$$
(\f_1,\dots,\f_p)\mapsto \one_K\bigwedge_i(\theta_i+ \ddc\f_i),
$$
acting on tuples of quasi-psh functions bounded on $K$ is local in the plurifine topology. 
\end{itemize}
Moreover, the construction is compatible with restriction, \ie 
\begin{equation}\label{equ:prodcomp}
\one_LT_1\winter T_p=\one_L\left(\one_KT_1\winter T_p\right)
\end{equation} 
for any compact subset $L\subset K$, and we further have 
\begin{equation}\label{equ:posprod}
T_1,\dots,T_p\ge 0\Longrightarrow\one_K T_1\winter T_p\ge 0. 
\end{equation}
\end{prop}

We emphasize that the notation $\one_K T_1\winter T_p$ above should be taken as a whole, \ie we do not try to make sense of $T_1\winter T_p$ itself (see however Theorem~\ref{thm:npp} below). 

\begin{proof} Pick $(1,1)$-forms $\theta_1,\dots,\theta_p$ and quasi-psh functions $\f_1,\dots,\f_p$ that are bounded on $K$. The latter is then contained in the plurifine open subset 
$$
O:=\bigcap_i\{\f_i>-t_0\}
$$
for some $t_0$. For any $t\ge t_0$ and $i=1,\dots,p$, the bounded quasi-psh function $\max\{\f_i,-t\}$ coincides with $\f_i$ on $O$. By (i), (ii), we thus necessarily have 
\begin{equation}\label{equ:prodex}
\one_K \bigwedge_i(\theta_i+ \ddc\f_i)=\one_K\bigwedge_i(\theta_i+ \ddc\max\{\f_i,-t\}),
\end{equation}
for $t$ large enough, where the right-hand side was defined above (in the sense of Bedford--Taylor). This proves uniqueness. To get existence, we use~\eqref{equ:prodex} as a definition, and first check that it only depends on the quasi-closed, quasi-positive currents $T_i=\theta_i+\ddc\f_i$, $i=1,\dots,p$. Any other decomposition of $T_i$ is of the form 
$$
T_i=(\theta_i+\ddc\tau_i)+\ddc(\f_i-\tau_i)
$$ 
with $\tau_i\in C^\infty(X)$ (see Lemma~\ref{lem:quasi}). For $t$ large enough, the bounded quasi-psh functions 
$$
\max\{\f_i-\tau_i,-t\},\quad\max\{\f_i,-t\}-\tau_i
$$
are both equal to $\f_i-\tau_i$ on the plurifine open $O$. By plurifine locality, this yields 
$$
\one_K\bigwedge_i\left(\left(\theta_i+\ddc\tau_i\right)+ \ddc\max\{\f_i-\tau_i,-t\}\right)=\one_K\bigwedge_i \left(\left(\theta_i+ \ddc\tau_i\right)+ \ddc\left(\max\{\f_i,-t\}-\tau_i\right)\right), 
$$
which coincides in turn with $\one_K\bigwedge_i \left(\theta_i+ \ddc\max\{\f_i,-t\}\right)$. By~\eqref{equ:prodex} we infer 
$$
\one_K \bigwedge_i\left((\theta_i+\ddc\tau_i)+ \ddc(\f_i-\tau_i)\right)=\one_K \bigwedge_i(\theta_i+ \ddc\f_i),
$$
which proves, as desired, that this current only depends on the $T_i$'s. Similar considerations shows that 
$\one_K T_1\winter T_p$ is a multilinear function of the $T_i$. Symmetry is trivial, and~\eqref{equ:prodcomp} holds by uniqueness. 

The proof of~\eqref{equ:posprod} is slightly more involved. Assume $T_i\ge 0$ for all $i$. Since
$$
\one_K T_1\winter T_p\ge 0
$$ 
is a local property, it is enough to argue in a small open neighborhood $\Om$ of a given point of $X$, on which we can thus find a K\"ahler form $\om= \ddc\rho$ with $\rho$ smooth and bounded. Pick $\e>0$. By Lemma~\ref{lem:approxclosed}, each point of $\Om$ admits an open neighborhood $U\subset\Om$ (depending on $\e$) and closed $(1,1)$-forms $\theta_i'$ on $U$ such that $-\e\om\le\theta_i-\theta'_i\le\e\om$ on $U$. After shrinking $U$, we further assume $\theta'_i= \ddc\tau_i$ with $\tau_i$ smooth and bounded on $U$. Then 
$$
 \ddc(\tau_i+\e\rho+\f_i)=\theta_i'+\e\om+ \ddc\f_i\ge\theta_i+ \ddc\f_i\ge 0,
$$
and hence $\tau_i+\e\rho+\f_i$ is psh on $U$. For each $t>0$, $\max\{\tau_i+\e\rho+\f_i,-t\}$ is thus psh and bounded on $U$, and it further coincides with the bounded quasi-psh function $\tau_i+\e\rho+\max\{\f_i,-t\}$ on the plurifine open $O\cap U$ for $t$ large enough. By plurifine locality of Bedford--Taylor products,  we get for each $I\subset\{1,\dots,p\}$
$$
\one_{K\cap U}\bigwedge_{i\in I}\left(\theta'_i+\e\om+ \ddc\f_i\right)=\one_{K\cap U}\bigwedge_{i\in I} \ddc\left(\tau_i+\e\rho+\max\{\f_i,-t\}\right)
$$
$$
=\one_{K\cap U}\bigwedge_{i\in I} \ddc\max\{\tau_i+\e\rho+\f_i,-t\}\ge 0.
$$
Since $\theta_i+\e\om\ge\theta'_i$ on $U$, we infer 
$$
\one_{K\cap U}\bigwedge_{i=1}^p(\theta_i+2\e\om+ \ddc\f_i)=\sum_{I\subset\{1,\dots,p\}}\one_{K\cap U}\bigwedge_{i\in I}(\theta'_i+\e\om+ \ddc\f_i)\wedge\bigwedge_{i\notin I}(\theta_i+\e\om-\theta'_i)\ge 0.
$$
As this holds for $U$ ranging over an open cover of $\Om$,  we infer 
$$
\one_{K\cap\Om}\bigwedge_{i=1}^p(\theta_i+2\e\om+ \ddc\f_i)\ge 0.
$$
This holds for any $\e>0$, and this current converges weakly to $\one_{K\cap\Om} \bigwedge_i(\theta_i+ \ddc\f_i)$ as $\e\to 0$, as one easily sees using multilinearity. We conclude $\one_K \bigwedge_i(\theta_i+ \ddc\f_i)\ge 0$ on $\Om$, which finishes the proof of~\eqref{equ:posprod}. 
\end{proof}

%
\subsection{Non-pluripolar products}
In this Section and \ref{sect:monomass}
 \emph{we assume that the manifold $X$ has the
bounded mass property}  (see~\S\ref{sec:bdmass}).  
Generalizing~\cite[\S 1]{BEGZ10} (which dealt with closed currents on a K\"ahler manifold), we then show: 

\begin{thm}\label{thm:npp} For any $p=1,\dots,n$, there exists a unique symmetric, multilinear pairing 
$$
(T_1,\dots,T_p)\mapsto  T_1\winter T_p 
$$
defined on $p$-tuples of quasi-closed, quasi-positive currents and with values in nonpluripolar $(p,p)$-currents, such that:
\begin{itemize}
\item[(i)] if the $T_i$'s have bounded potentials, then the pairing is compatible with~\eqref{equ:wedgebd}; 
\item[(ii)] for any tuple of $(1,1)$-forms $\theta_1,\dots,\theta_p$, the operator 
$$
(\f_1,\dots,\f_p)\mapsto \bigwedge_i(\theta_i+ \ddc\f_i), 
$$
defined on all tuples of quasi-psh functions, is local in the plurifine topology. 
\end{itemize}
Moreover, 
$$
T_1,\dots,T_p\ge 0\Longrightarrow  T_1\winter T_p \geq 0. 
$$
\end{thm}
The current $T_1\winter T_p$ is called the \emph{non-pluripolar product} of the $T_i$'s. When $\theta_i=\theta$ and $\varphi_i=\varphi$, we write
$$
(\theta+dd^c \varphi)^p := \bigwedge_i(\theta_i+ \ddc\f_i). 
$$
The key input in the proof is the following: 

\begin{lem}\label{lem:bdmass} 
Let $T_1,\dots,T_p$ be quasi-closed, quasi-positive currents, and pick a compact subset $K\subset X$ on which the $T_i$'s have bounded potentials. Then the total variation of the current $\one_KT_1\winter T_p$ defined in Proposition~\ref{prop:prodbd} is bounded by a constant independent of $K$. 

More precisely, write $T_i=\theta_i+ \ddc\f_i$ with $\theta_i$ smooth and $\f_i$ quasi-psh, and assume 
$$
\pm\theta_i\le C\om_X\quad\text{and}\quad \ddc\f_i\ge-C\om_X
$$
for some $C>0$ and
for all $i$. Then the total variation of $\one_KT_1\winter T_p$ is bounded by a constant only depending on $\om_X$ and $C$. 
\end{lem}

\begin{proof} 
By multilinearity, we may assume without loss of generality that
$\theta_i=\om_X$ and $T_i\ge 0$, \ie $\f_i\in\PSH(\om_X)$. By~\eqref{equ:posprod}, we then have $\one_KT_1\winter T_p\ge 0$, and
$$
\int \one_KT_1\winter T_p\wedge\om_X^{n-p}\le n^n\int \one_K(\om_X+ \ddc\p)^n
$$
with $\p:=\tfrac 1n\sum_{i=1}^p\f_i$, by multilinearity. By construction, the right-hand side is unchanged when we replace $\p$ with the bounded $\om_X$-function $\max\{\p,-t\}$ for $t\gg 1$ (see Proposition~\ref{prop:prodbd}), and 
$$
\int \one_K(\om_X+ \ddc\p)^n=\int \one_K(\om_X+ \ddc\max\{\p,-t\})^n
$$
$$
\le\int_X(\om_X+ \ddc\max\{\p,-t\})^n\le\uvol(\om_X)<\infty, 
$$ 
see~\S\ref{sec:bdmass}. 
\end{proof}

\begin{proof}[Proof of Theorem~\ref{thm:npp}] Write $T_i=\theta_i+ \ddc\f_i$ with $\theta_i$ smooth and $\f_i$ quasi-psh. For each $m\in\N$, the $T_i$'s have bounded potentials on the compact subset $K_m:=\bigcap_i\{\f_i\ge-m\}$, and $(K_m)_m$ is an increasing sequence such that 
$$
X\setminus\bigcup_m K_m=\bigcup_i\{\f_i=-\infty\}
$$
is pluripolar. For each $m$, Proposition~\ref{prop:prodbd} yields a nonpluripolar $(p,p)$-current
$$
\Theta_m:=\one_{K_m} T_1\winter T_p,
$$
which further satisfies $\Theta_m=\one_{K_m} \Theta_{m+1}$. By Lemma~\ref{lem:bdmass}, the total variation of $\Theta_m$ is further bounded independently of $m$, and it follows that there exists a unique nonpluripolar $(p,p)$-current $\Theta$ on $X$ such that $\one_{K_m}\Theta=\Theta_m$ for all $m$. By (i), (ii), this current $\Theta$ necessarily coincides with $T_1\winter T_p$, and the rest now follows from Proposition~\ref{prop:prodbd}.
\end{proof}

We further record the following consequence of Lemma~\ref{lem:bdmass}: 

\begin{lem}\label{lem:bdmass2} For all quasi-closed positive currents $T_1,\dots,T_n\ge 0$, we have 
$$
\int_X T_1\winter T_n\le M
$$
for a constant $M>0$ only depending on the psef classes $\{T_i\}\in\BC^{1,1}(X,\R)$. 
\end{lem}

\subsection{Monotonicity of non-pluripolar masses}\label{sect:monomass}
Arguing as in~\cite[Theorem~2.3]{DDL}, we establish the following semicontinuity property of non-pluripolar Monge--Amp\`ere masses. 
 
\begin{prop}\label{prop:lscmass} 
For any $(1,1)$-forms $\theta_1,\dots,\theta_n$, 
$$
\int_X (\theta_1+dd^c\f_1)\winter(\theta_n+dd^c\f_n)
$$
is a lower semicontinuous function of the tuple $\f_i\in\PSH(\theta_i)$, $i=1,\dots,n$, with respect to convergence in capacity. 
\end{prop}

We refer to~\cite[Chapter~4]{GZbook} for the notion of convergence in capacity. Suffice it to say here that any monotonically convergent sequence of $C\om_X$-psh functions converges in capacity. The proof of Proposition~\ref{prop:lscmass} relies on the 
following continuity property of Monge--Amp\`ere integrals, which is a direct consequence of the corresponding local statement,
see \cite[Theorem~4.25]{GZbook}.   

\begin{lem}\label{lem:capcont} 
Pick $C>0$, and assume $\f_i^k\in\PSH(C\om_X)$ is uniformly bounded and converges in capacity to $\f_i\in\PSH(C\om_X)$ as $k\to\infty$, for $i=1,\dots,n$. Let also $(f_k)$ be a bounded sequence of quasi-continuous functions, converging in capacity to a quasi-continuous function $f$. Then 
$$
\lim_k \int_X f_k\bigwedge_i(\theta_i+ \ddc\f_i^k)=\int_X f\bigwedge_i(\theta_i+ \ddc\f_i). 
$$
\end{lem}
Recall that a function $f$ is \emph{quasi-continuous} if is continuous outside open subsets of arbitrarily small capacity. Any quasi-psh function is quasi-continuous, and we thus have: 

\begin{exam}\label{exam:qcont} Pick finitely many quasi-psh functions $\f_i$, and consider the plurifine open set $O:=\{\f>0\}$ with $\f:=\min_i\f_i$. For each $\e>0$, the function $\chi_\e\colon X\to [0,1]$ defined by 
$$
\chi_\e:=\frac{\max\{\f,0\}}{\max\{\f,0\}+\e}
$$
is then quasi-continuous, and $\chi_\e$ increases pointwise to $\one_O$ as $\e\searrow 0$. 
\end{exam}

\begin{proof}[Proof of Proposition~\ref{prop:lscmass}] For each $k$ and $t,\e>0$, consider as in Example~\ref{exam:qcont} the quasi-continuous approximations 
$$
\chi_{k,\e,t}:=\frac{\max\{\min_i\f_i^k+t,0\}}{\max\{\min_i\f_i^k+t,0\}+\e},\quad\chi_{\e,t}:=\frac{\max\{\min_i\f_i+t,0\}}{\max\{\min_i\f_i+t,0\}+\e}
$$
of the characteristic functions of the plurifine open sets
$$
O_{k,t}:=\bigcap_i\{\f_i^k>-t\},\quad O_t:=\bigcap_i\{\f_i>-t\}. 
$$
Since $\bigwedge_i(\theta_i+ \ddc\f_i^k)\ge 0$ and $\chi_{k,\e,t}$ vanishes outside $O_{k,t}$, plurifine locality of nonpluripolar products yields 
$$
\int_X \bigwedge_i(\theta_i+ \ddc\f_i^k) \ge\int_X\chi_{k,\e,t} \bigwedge_i(\theta_i+ \ddc\f_i^k) =\int_X\chi_{k,\e,t}\,\bigwedge_i(\theta_i+ \ddc\max\{\f_i^k,-t\}). 
$$
Now $(\chi_{k,\e,t})_k$ is a bounded sequence of quasi-continuous fonctions converging in capacity to the quasi-continuous function $\chi_{\e,t}$, while $(\max\{\f_i^k,-t\})_k$ is a bounded sequence of $C\om_X$-psh functions converging in capacity to $\max\{\f_i,-t\}$. By Lemma~\ref{lem:capcont}, we thus have
$$
\lim_k\int_X \chi_{k,\e,t}\,\bigwedge_i(\theta_i+ \ddc\max\{\f_i^k,-t\})=\int_X \chi_{\e,t}\bigwedge_i(\theta_i+ \ddc\max\{\f_i,-t\})
$$
$$
=\int_X \chi_{\e,t}\, \bigwedge_i(\theta_i+ \ddc\f_i) , 
$$
where the last equality holds by plurifine locality again. We thus get
$$
\liminf_k\int_X \bigwedge_i(\theta_i+ \ddc\f_i^k)\ge\int_X \chi_{\e,t} \bigwedge_i(\theta_i+ \ddc\f_i), 
$$
and the result follows since 
$$
\lim_{t\to\infty}\lim_{\e\to 0}\int_X \chi_{\e,t} \bigwedge_i(\theta_i+ \ddc\f_i)=\lim_{t\to\infty}\int_{O_t}\bigwedge_i(\theta_i+ \ddc\f_i)=\int_X\bigwedge_i(\theta_i+ \ddc\f_i).
$$
\end{proof}

Finally, we provide a version, in the present non-K\"ahler setting, of the monotonicity of nonpluripolar masses with respect to singularities. To state the result, pick any $(1,1)$-form $\theta$, let $C\ge 0$ be the smallest constant such that $\pm\theta\le C\om_X$, and set 
\begin{equation}\label{equ:Dthe}
\D_\theta:=\sup_{\f,\p\in\PSH(C\om_X)\cap L^\infty}\left\{\int_X(\theta+ \ddc\f)^n-\int_X(\theta+ \ddc\p)^n\right\}. 
\end{equation}
By Stokes' theorem, $\D_\theta=0$ if $\theta$ is closed. In general,  Lemma~\ref{lem:bdmass2} yields $\D_\theta\in[0,+\infty)$ and 
\begin{equation}\label{equ:Dvar}
\D_{\theta+\e\om_X}=\D_\theta+O(\e)
\end{equation}
for $\e\in [0,1]$. 

\begin{thm}\label{thm:monomass} Assume $\f,\p\in\PSH(\theta)$ satisfy $\f\le\p+O(1)$. Then 
$$
\int_X (\theta+ \ddc\f)^n \le\int_X (\theta+ \ddc\p)^n+\D_\theta. 
$$
\end{thm}
\begin{proof} The proof goes along the same lines as~\cite[Lemma~3.1]{DDL2}. Assume first $\f=\p$ outside some compact subset $K\subset X$ on which $\f,\p$ are bounded. Pick $t\gg 1$ such that $\f,\p>-t$ on $K$, and consider the bounded $C\om_X$-psh functions 
$$
\f_t:=\max\{\f,-t\},\quad\p_t:=\max\{\p,-t\}.
$$
By plurifine locality, we then have 
$$
\int_K(\theta+ \ddc\f)^n=\int_K(\theta+ \ddc\f_t)^n,\quad\int_K(\theta+ \ddc\p)^n=\int_K(\theta+ \ddc\p_t)^n,
$$
while $\f=\p,\f_t=\p_t$ on the open set $X\setminus K$ similarly yields
$$
\int_{X\setminus K}(\theta+ \ddc\f)^n=\int_{X\setminus K}(\theta+ \ddc\p)^n,\quad\int_{X\setminus K}(\theta+ \ddc\f_t)^n=\int_{X\setminus K}(\theta+ \ddc\p_t)^n.
$$
This implies 
$$
\int_X(\theta+ \ddc\f)^n-\int_X(\theta+ \ddc\p)^n=\int_X(\theta+ \ddc\f_t)^n-\int_X(\theta+ \ddc\p_t)^n,
$$
and the result follows since the right-hand side is bounded by $\D_\theta$. Returning to the general case, set 
$$
\f_{\e,t}:=\max\{(1+\e)\f,\p-t\}\in\PSH(\theta+C\e\om_X), 
$$
where $\e,t>0$ are fixed for the moment. Since $\f\le\p+O(1)$, we have $\f_{\e,t}=\p-t$ outside the compact set $\{\f\ge -s\}$ for $s\gg 1$, and the first part of the proof thus yields
$$
\int_X(\theta+C\e\om_X+ \ddc\f_{\e,t})^n\le\int_X(\theta+C\e\om_X+ \ddc\p)^n+\D_{\theta+C\e\om_X}. 
$$
For $0<\e\le\d$, we infer 
$$
\int_X(\theta+C\d\om_X+ \ddc\f_{\e,t})^n=\int_X(\theta+C\e\om_X+\ddc\f_{\e,t})^n+O(\d)
$$
$$
\le\int_X(\theta+C\e\om_X+ \ddc\p)^n+\D_\theta+O(\d),
$$
using Lemma~\ref{lem:bdmass} and~\eqref{equ:Dvar}. Since $\f_{t,\e}\in\PSH(\theta+C\d\om_X)$ increases to $\max\{\f,\p-t\}$ as $\e\to 0$, Proposition~\ref{prop:lscmass} now yields 
$$
\int_X(\theta+C\d\om+ \ddc\max\{\f,\p-t\})^n\le\int_X(\theta+ \ddc\p)^n+\D_\theta+O(\d). 
$$
Letting $\d\to 0$, we get
$$
\int_X(\theta+ \ddc\max\{\f,\p-t\})^n\le\int_X(\theta+ \ddc\p)^n+\D_\theta. 
$$
Since $\max\{\f,\p-t\}$ decreases to $\f$ as $t\to\infty$,  the conclusion now
follows from Proposition~\ref{prop:lscmass}. 
\end{proof}

\subsection{Monge--Amp\`ere measures of quasi-psh envelopes}
In this section, we fix a big $dd^c$-class $\{\theta\} \in \BC^{1,1}(X,\mathbb R)$. We also pick $\rho \in \PSH(\theta)$ with analytic singularities such that $\theta+dd^c \rho \geq \omega_X$ (after perhaps scaling $\om_X$), and consider the Zariski open subset 
$$
\Omega:=\{\rho>-\infty\}.
$$
\begin{thm}\label{thm:envelope}
Pick a quasi-continuous function $f\colon X\to[-\infty,\infty]$, and assume that the usc regularization $\env_\theta^\star(f)$ of its $\theta$-psh envelope~\eqref{equ:pshenv} is $\theta$-psh. Then
$$
\int_{\{\env_\theta^\star(f)<f\}} (\theta+dd^c \env_\theta^\star(f))^n =0. 
$$
\end{thm}

\begin{proof} Replacing $f$ with $\min\{f,\sup_X\env_\theta^\star(f)\}$, we may assume without loss that $f$ is bounded above. Assume further that $f$ is bounded. When $f$ is lsc, the result follows from a standard balayage argument. If $f$ is merely bounded, we can find a decreasing sequence $(f_j)$ of bounded lsc functions that decreases pointwise to $f$ outside a pluripolar set (see~\cite[Lemma 2.4]{GLZ19} and~\cite[Theorem 2.7]{DDL2}). Using that $u_j:= \env^\star_\theta(f_j)$ (resp.~$u:=\env_\theta^\star(f)$) is the largest $\theta$-psh function dominated by $f_j$ (resp.~$f$) outside a pluripolar set, it is easy to see that $u_j$ decreases pointwise to $u$. By the previous step of the proof, we further have
$$	
\int_X \frac{\max\{f_j-u_j,0\}}{\max\{f_j-u_j,0\}+\e} (\theta+dd^c u_j)^n =0 
$$
for each $\e>0$. Letting $j\to +\infty$, and arguing as in Proposition \ref{prop:lscmass}, we obtain 
$$
\int_X \frac{\max\{f-u,0\}}{\max\{f-u,0\}+\e} (\theta+dd^c u)^n =0, 
$$
and letting $\e\to 0$ yields the result in the bounded case. If $f$ is not bounded, we approximate it by $f_j=\min(f,-j)$, and proceed as above. 
\end{proof}

\begin{rmk}
When $\env_\theta^\star(f)$ is locally bounded in $\Omega$, the measure $(\theta+dd^c \env_\theta^\star(f))^n$ is understood in the classical sense of Bedford--Taylor in $\Omega$, and can thus be viewed as a nonpluripolar measure on $X$ without assuming the bounded mass property. The same remark applies for the results below. 
\end{rmk}

\begin{prop}\label{prop:maxprinciple}
	Assume $\{\theta\} \in \BC^{1,1}(X,\mathbb R)$ is big. If $u,v\in \PSH(\theta)$ then 
	$$
	(\theta+dd^c \max\{u,v\})^n \geq \one_{\{u\ge v\}}(\theta+dd^c u)^n + \one_{\{v>u\}}(\theta+dd^c v)^n. 
	$$
	In particular, if $u\leq v$, then 
	$$
	\one_{\{u=v\}}(\theta+dd^c u)^n \leq \one_{\{u=v\}}(\theta+dd^c v)^n. 
	$$
\end{prop}

\begin{proof}
	If $u$ and $v$ have minimal singularities, then they are locally bounded in $\Omega$. By~\cite[Lemma 1.2]{GL22}, the result thus holds on $\Om$, and hence on $X$ since all measures involved are nonpluripolar. In the general case, we approximate $u,v$ by the $\theta$-psh functions with minimal singularities
	$$
	u_t:=\max\{u,V_{\theta}-t\},\quad v_t:=\max\{v,V_\theta-t\}.
	$$
	The previous step yields 
	$$
	(\theta+dd^c \max\{u_t,v_t\})^n \geq \one_{\{v_t\leq u_t\}}(\theta+dd^c u_t)^n + \one_{\{u_t<v_t\}}(\theta+dd^c v_t)^n. 
	$$
	Multiplying with $\one_{\{\min\{u,v\} >V_{\theta}-t\}}$,  we get 
	\begin{flalign*}
		\one_{\{\min\{u,v\} >V_{\theta}-t\}}(\theta+dd^c \max\{u,v\})^n &\leq \one_{\{\min\{u,v\} >V_{\theta}-t\}}\one_{\{v\leq u\}}(\theta+dd^c u)^n \\
		&+ \one_{\{\min\{u,v\} >V_{\theta}-t\}}\one_{\{u<v\}}(\theta+dd^c v)^n. 
	\end{flalign*}
	We finally let $t\to +\infty$ to obtain the result. 
\end{proof}

As a consequence we infer the following minimum principle.

\begin{cor}\label{cor:minprinciple}
	Pick $u,v \in \PSH(\theta)$ such that there exists $w\in\PSH(\theta)$ with $w\le u,v$, and hence $\env_\theta(\min\{u,v\})\in \PSH(\theta)$. Then 
	\[
	(\theta+dd^c \env_\theta(\min\{u,v\}))^n \leq \one_{\{\env_\theta(\min\{u,v\})=v\leq u\}}(\theta+dd^c v)^n + \one_{\{\env_\theta(\min\{u,v\})=u<v\}}(\theta+dd^c u)^n.
	\] 
\end{cor}

%

\section{The lower and upper volume of a $\ddc$-class} \label{sec:vol}

As above, $X$ denotes an $n$-dimensional compact complex manifold, 
assumed throughout this section to have the bounded mass property. We introduce and study the notion of lower and upper volume of a $dd^c$-class. 

%

\subsection{The lower and upper volume of a current} 
Extending~\cite[Definition 3.1]{GL22}, we introduce:

\begin{defi} The \emph{lower volume} and \emph{upper volume} of a quasi-closed positive $(1,1)$-current $T\ge 0$ are defined as  
$$
\lvol(T):=\inf_S\int_X S^n,\quad\uvol(T):=\sup_S\int_X S^n, 
$$
where $S$ ranges over all quasi-closed positive $(1,1)$-currents in the same $\ddc$-class as $T$ and with equivalent singularities. 
\end{defi}
Writing $T=\theta+\ddc\f$ with $\theta$ smooth and $\f\in\PSH(\theta)$, we equivalently have 
\begin{equation}\label{equ:volcurr}
\lvol(T)=\inf_\p\int_X(\theta+\ddc\p)^n,\quad\uvol(T)=\sup_\p\int_X(\theta+\ddc\p)^n
\end{equation}
where $\p$ ranges over all $\theta$-psh functions such that $\p=\f+O(1)$. 

By definition, the upper and lower volumes of $T$ only depend on its $\ddc$-class and singularity class, and satisfy 
$$
0\le\lvol(T)\le\uvol(T)<\infty, 
$$
see Lemma~\ref{lem:bdmass2}. Note also that 
\begin{equation}\label{equ:uvolpos}
T>0\Longrightarrow \uvol(T)>0. 
\end{equation}
The analogous result for $\lvol(T)$ is unclear in general, even when $T$ is smooth (see~\S\ref{sec:posvol}). As a consequence of Theorem~\ref{thm:monomass}, we have:  

\begin{exam}\label{exam:vclosed}
 If $T$ is closed, then $\lvol(T)=\uvol(T)=\int_X T^n$.
\end{exam}

We first collect a few simple properties:

\begin{prop}\label{prop:lvol} 
Pick a quasi-closed current $T\ge 0$. 
Then:
\begin{itemize}
\item[(i)] for all $t\in\R_{>0}$ we have 
$$
\lvol(tT)=t^n\lvol(T),\quad \uvol(tT)=t^n\uvol(T);
$$
\item[(ii)] for any quasi-closed current $T'\ge 0$ and $\e\in[0,1]$, we have 
\begin{equation}\label{equ:vsum}
\lvol(T+\e T')\le \lvol(T)+C\e 
\end{equation}
for a constant $C>0$ only depending on $\{T\},\{T'\}$;
\item[(iii)] for any effective $\R$-divisor $E$ we have 
$$
\lvol(T+[E])=\lvol(T),\quad\uvol(T+[E])=\uvol(T);
$$
\item[(iv)] for any modification $\pi\colon Y\to X$ we have 
$$
\lvol(\pi^\star T)=\lvol(T),\quad\uvol(\pi^\star T)=\uvol(T).
$$
\end{itemize}
\end{prop}

\begin{proof} (i) is obvious from the definition. Pick $T'\ge 0$ quasi-closed and $0\le S\in \{T\}$ with singularities equivalent to $T$. Then $0\le S+\e T'\in \{T+\e T'\}$ has singularities equivalent to $T+\e T'$, and hence $\int_X(S+\e T')^n\ge \lvol(T+\e T')$. For $j=0,\dots,n$, $\int_X S^j\wedge T'^{n-j}$ is further bounded in terms of $\{S\}=\{T\}$ and $\{T'\}$, by Lemma~\ref{lem:bdmass2}. Thus 
$$
\int_X S^n\ge\int_X(S+\e T')^n-C\e\ge \lvol(T+\e T')-C\e
$$
for a constant $C>0$ only depending on $\{T\},\{T'\}$.
Taking the infimum over $S$ yields (ii). 

Next pick an effective $\R$-divisor $E$, and note that any $0\le S\in\{T\}$ satisfies $S+[E]\in\{T+[E]\}$ and $\int_XS^n=\int_X(S+[E])^n$. To prove (iii), it is thus enough to show that $S\mapsto S+[E]$ sets a 1--1 correspondence between positive currents in $\{T\}$ with the same singularities as $T$ and positive currents in $\{T+[E]\}$ with the same  singularities as $T+[E]$. One direction is clear. Conversely, pick $0\le S'\in\{T+[E]\}$ with the same singularities as $T+[E]$. Write 
$$
T=\theta+\ddc\p,\quad [E]=\theta_E+\ddc\p_E,\quad S'=\theta+\theta_E+\ddc\f'
$$ 
with $\theta,\theta_E$ smooth and $\p\in\PSH(\theta)$, $\p_E\in\PSH(\theta_E)$, $\f'\in\PSH(\theta+\theta_E)$. By assumption we have $\f'=\p+\p_E+O(1)$. In particular, $\f'-\p_E$ is $\theta$-psh on $X\setminus\supp E$ and bounded above, and hence uniquely extends to a $\theta$-psh function $\f$ on $X$, which satisfies $\f=\p+O(1)$. This shows that $S:=\theta+\ddc\f\ge 0$ lies in $\{T\}$, has the same singularities as $T$, and satisfies $S'=S+[E]$, which concludes the proof of (iii). 

Finally, pick a modification $\pi\colon Y\to X$. Since $\pi$ has connected fibers, each positive current in $\{\pi^\star T\}$ is of the form $\pi^\star S$ with $0\le S\in\{T\}$, and $S$ has the same singularities as $T$ iff $\pi^\star S$ has the same singularities as $\pi^\star T$. Since $\int_X S^n=\int_Y(\pi^\star S)^n$, we get (iv). 
\end{proof}

\begin{rmk} The analogue of~\eqref{equ:vsum} for $\uvol$ holds when $T>0$ (with a constant depending on $T$), but is unclear in general. 
\end{rmk}

We shall in fact mostly consider the upper and lower volumes of currents with analytic singularities (see~\S\ref{sec:qclosed}), which basically reduces to the smooth case thanks to the following consequence of Proposition~\ref{prop:lvol}~(iii),(iv). 

\begin{lem}\label{lem:vres} Assume $T\ge 0$ is quasi-closed with analytic singularities, and pick a modification $\pi\colon Y\to X$ such that $\pi^\star T=\b+[E]$ with $\b\ge 0$ smooth and $E$ an effective $\Q$-divisor. Then $\lvol(T)=\lvol(\b)$ and $\uvol(T) = \uvol(\b)$. 
\end{lem}
Building on~\cite{GL22}, we infer the following monotonicity property. 

\begin{prop}\label{prop:vmono} 
Let $T_1,T_2\ge 0$ be quasi-closed positive currents with analytic singularities. 
If $T_1\le T_2$, then $\lvol(T_1)\le\lvol(T_2)$ and $\uvol(T_1) \le \uvol(T_2)$. 
\end{prop}

\begin{proof} 
After passing to a joint resolution of singularities of $T_1$ and $T_2$ we may assume $T_i=\b_i+[E_i]$ with $\b_i\ge 0$ smooth and $E_i$ an effective $\Q$-divisor, and we then have $\lvol(T_i)=\lvol(\b_i)$ (see Lemma~\ref{lem:vres}). Then $T_1\le T_2$ implies $\b_1+[E_1]\le\b_2+[E_2]$, and hence $\b_1\le\b_2$. By~\cite[Proposition~3.2]{GL22} we get $\lvol(\b_1)\le\lvol(\b_2)$ and $\uvol(\b_1)\le\uvol(\b_2)$, which concludes the proof. 
\end{proof}

While we are not able at the moment to extend this result to arbitrary quasi-closed positive currents, quasi-psh envelopes enable us the following related monotonicity result in the general case. 

\begin{prop}\label{prop:vmono2}
Fix smooth real $(1,1)$-forms $\theta_1,\theta_2$ and $\varphi_1\in \PSH(\theta_1)$, 
$\varphi_2\in \PSH(\theta_2)$ with $\varphi_1\leq \varphi_2$ and $\theta_1\leq \theta_2$.  
Then the positive currents $T_1=\theta_1+dd^c \varphi_1$, $T_2=\theta_2+dd^c \varphi_2$ satisfy
$$
\lvol (T_1)\leq \lvol(T_2),\quad\uvol(T_1)\leq \uvol(T_2).
$$ 	
\end{prop}
In particular, the upper and lower volumes of any quasi-closed positive current $T$, which only depend on its $dd^c$-class and singularity class, are further monotone increasing functions of the singularity class.

\begin{proof}
Pick $\p_2\in \PSH(\theta_2)$ such that $\p_2=\varphi_2+O(1)$, and set 
$$
\f_t :=\env_{\theta_1}(\min\{\varphi_1+t,\p_2\})
$$
for $t>0$. Using $\f_1\le\f_2=\p_2+O(1)$ we get $\f_t\in\PSH(\theta_1)$, $\f_t=\f_1+O(1)$, and hence $\lvol(T_1)\le\int_X(\theta_1+dd^c\f_t)^n$. By Theorem \ref{thm:envelope}, we further have
	\begin{flalign*}
		(\theta_1+dd^c \f_t)^n & \leq \one_{\{\f_t=\varphi_1+t\}} (\theta_1+dd^c \varphi_1)^n + \one_{\{\f_t=\p_2\}} (\theta_2+dd^c v_t)^n\\
		&\leq \one_{\{\varphi_1+t\leq \p_2\}} (\theta_1+dd^c \varphi_1)^n + \one_{\{\f_t=\p_2\}} (\theta_2+dd^c \p_2)^n\\
		& \leq  \one_{\{\varphi_1\leq \p_2-t\}} (\theta_1+dd^c \varphi_1)^n + (\theta_2+dd^c \p_2)^n. 
	\end{flalign*}
	Integrating over $X$ and letting $t\to +\infty$, we obtain $\lvol(T_1) \leq \int_X (\theta_2+dd^c \p_2)^n$. Taking the infimum over $\p_2$ we obtain the first inequality. 
	 
	Next pick $\p_1\in \PSH(\theta_1)$ such that $\p_1=\varphi_1+O(1)$ is bounded, and set 
	$$
	\p_t := \max\{\p_1,\varphi_2-t\}
	$$
	for  $t>0$. Then $\p_t\in\PSH(\theta_2)$, $\p_t=\f_2+O(1)$, and $\p_t \searrow \p_1$ as $t\to +\infty$. By lower-semicontinuity of non-pluripolar masses (see Proposition~\ref{prop:lscmass}), we infer
	\[
	\int_X (\theta_2+dd^c  \p_1)^n\le\liminf_{t\to +\infty} \int_X (\theta_2+dd^c \p_t)^n\le \uvol(T_2) 
	\]
	Taking supremum over all $\p_1$ concludes the proof. 
\end{proof}

%

\subsection{The lower and upper volume of a $dd^c$-class} 

%
\subsubsection{The case of a big $\ddc$-class} 

Recall from~\S\ref{sec:posclass} that a $dd^c$-class $\{\theta\}\in\BC^{1,1}(X,\R)$ is big if it contains a strictly positive current, which can further be chosen to have analytic singularities (see Theorem \ref{thm:Demreg}).

\begin{defi}\label{defi:volbig} The \emph{lower volume} and \emph{upper volume} of a big $\ddc$-class $\{\theta\}\in\BC^{1,1}(X)$ are respectively defined as 
\begin{equation}\label{equ:volclass}
\lvol(\{\theta\}):=\sup_{0\le T\in\{\theta\}}\lvol(T),\quad \uvol(\{\theta\}):=\sup_{0\le T\in\{\theta\}}\uvol(T). 
\end{equation}
where $T$ ranges over all positive currents in $\{\theta\}$. 
\end{defi}
Note that 
$$
0\le\lvol(\{\theta\})\le\uvol(\{\theta\})=\sup_{\f\in\PSH(\theta)}\int_X(\theta+dd^c\f)^n<\infty, 
$$
see~\eqref{equ:volcurr} and Lemma~\ref{lem:bdmass2}. By~\eqref{equ:uvolpos}, we have 
$$
\{\theta\}\text{ big }\Longrightarrow\uvol(\{\theta\})>0,
$$
see~\S\ref{sec:posvol} for a discussion of the analogous property for $\lvol$. 

\begin{lem}\label{lem:volmin} For any positive current $T_{\min}\in\{\theta\}$ with minimal singularities, we have 
$$
\lvol(\{\theta\})=\lvol(T_{\min}),\quad\uvol(\{\theta\})=\uvol(T_{\min}).
$$
If $\theta$ is further closed, then 
$$
\lvol(\{\theta\})=\uvol(\{\theta\})=\int_X T_{\min}^n.
$$
\end{lem}
The last point ensures compatibility of Definition~\ref{defi:volbig} with~\cite{B02,BEGZ10} for usual (closed) Bott--Chern classes $\{\theta\}\in H^{1,1}_\BC(X,\R)$. 

\begin{proof} By Proposition~\ref{prop:vmono2}, the upper and lower volume of $T\in\{\theta\}$ are both monotone increasing functions of the singularity class. This proves the first point, and the second one now follows from Theorem~\ref{thm:monomass} (see Example~\ref{exam:vclosed}).
\end{proof}

\begin{lem}\label{lem:volanal}
For any big class $\{\theta\}\in \BC^{1,1}(X,\R)$, we have 
$$
\lvol(\{\theta\})=\sup_T\lvol(T),\quad\uvol(\{\theta\})=\sup_T\uvol(T)
$$
where $T$ ranges over all positive currents in $\{\theta\}$ with analytic singularities.
\end{lem}

\begin{proof} Pick $\f_{\min}\in\PSH(\theta)$ with minimal singularities, and $\rho \in \PSH(\theta)$ with analytic singularities such that $\theta+dd^c \rho\geq\omega_X$ (after perhaps scaling $\om_X$). By Theorem \ref{thm:Demreg}, there exists a sequence $\f_j\in \PSH(\theta+\e_j \omega_X)$ with analytic singularities decreasing to $\f_{\min}$, where $\e_j \searrow 0_+$. Set $\p_j:=(1+\e_j)^{-1}(\f_j+\e_j\rho)$ and $T_j:=\theta+dd^c\p_j$. Then $0\le T_j\in\{\theta\}$ has analytic singularities, and $$
\theta+\e_j \omega_X + dd^c \varphi_j \leq (1+\e_j)T_j
$$
By Proposition~\ref{prop:vmono} and Proposition~\ref{prop:vmono2}, this yields
$$
\lvol(\{\theta\})=\lvol(\theta+dd^c \f_{\min}) \leq \lvol(\theta+\e_j \omega_X + dd^c \varphi_j)\leq (1+\e_j)^n \lvol(T_j). 
$$
Thus $\limsup_j\lvol(T_j)\ge\lvol(\{\theta\})$, and similarly for $\uvol$. 
\end{proof}

\begin{prop}\label{prop:volbig} For all big classes $\{\theta\},\{\theta'\}\in\BC^{1,1}(X,\R)$, we have 
$$
\theta\le\theta'\Longrightarrow\lvol(\{\theta\})\le\lvol(\{\theta'\}),\quad\uvol(\{\theta\})\le\uvol(\{\theta'\}).
$$
Furthermore, the upper and lower volume functions are continuous on the big cone of $\BC^{1,1}(X,\R)$.
 \end{prop}

\begin{proof} The first point follows from Proposition \ref{prop:vmono2}. Pick a convergence sequence of big classes $\{\theta_j\}\to\{\theta\}$, so that $\theta_j\to\theta$ smoothly for some choice of representatives. Then $-\e_j\om_X\le\theta_j-\theta\le\e_j\om_X$ with $\e_j\to 0$. By the first point,  
it therefore suffices to show 
$$
\lvol(\{\theta\pm\e\om_X\})\to\lvol(\{\theta\}),\quad\uvol(\{\theta\pm\e\om_X\})\to\uvol(\{\theta\})
$$
as $\e\to 0$. To see this, pick a strictly positive current with analytic singularities $S\in\{\theta\}$. We may assume $S\ge\om_X$. For any $0<\e<1$ and $T\in\{\theta+\e\om_X\}$ with analytic singularities, 
$$
T_\e:=(1-\e)(T-\e\om_X)+\e S=(1-\e)T+\e (S-\om_X)+\e^2\om_X
$$
is then a positive current with analytic singularities in $\{\theta\}$. Since $T_\e\ge (1-\e)T$, Proposition~\ref{prop:vmono} yields $\lvol(\{\theta\})\ge\lvol(T_\e)\ge (1-\e)^n \lvol(T)$. Taking the supremum over $T$ we get
$$
\lvol(\{\theta+\e\om_X\})\ge\lvol(\{\theta\})\ge(1-\e)^n\lvol(\{\theta+\e\om_X\}),
$$
by Lemma~\ref{lem:volanal}, and hence $\lvol(\{\theta+\e\om_X\})\to\lvol(\{\theta\})$. To prove $\lvol(\{\theta-\e\om_X\})\to\lvol(\{\theta\})$ we pick a positive current with analytic singularities $T\in \{\theta\}$. Then 
$$
S_{\e}:= (1-\e)T + \e S - \e \omega_X \geq (1-\e)T \geq 0
$$
is a positive current with analytic singularities in $\{\theta-\e \omega_X\}$. By Proposition \ref{prop:vmono}, we infer $\lvol(\{\theta-\e \omega_X\}) \geq (1-\e)^n \lvol(T)$. Taking the supremum over $T$ yields
$$
\lvol(\{\theta\})\ge\lvol(\{\theta-\e \omega_X\})\ge(1-\e)^n\lvol(\{\theta\}),
$$
thus $\lvol(\{\theta-\e\om_X\})\to\lvol(\{\theta\})$.
The proof of $\uvol(\{\theta\pm \e\om_X\})\to\uvol(\{\theta\})$ is similar.
\end{proof} 

%

We conclude this section with the following Fujita-type approximation result (compare~\cite[Theorem~1.4]{B02}). 

\begin{prop}\label{prop:Fuj} For any big class $\{\theta\}\in \BC^{1,1}(X,\R)$ we have 
$$
\lvol(\{\theta\})=\sup_Y\lvol(\om_Y),\quad\uvol(\{\theta\})=\sup_Y\uvol(\om_Y)
$$
where the suprema range over all modifications $\pi\colon Y\to X$ and decompositions
$$
\{\pi^\star\theta\}=\{\om_Y\}+\{E\}
$$
with $\om_Y>0$ smooth and $E$ an effective $\Q$-divisor. 
\end{prop}
\begin{proof} Pick a modification $\pi\colon Y\to X$ and a decomposition
$$
\{\pi^\star\theta\}=\{\om_Y\}+\{E\}
$$
with $\om_Y>0$ smooth and $E$ an effective $\Q$-divisor. Then Proposition~\ref{prop:volmon} yields 
$$
\lvol(\om_Y)=\lvol(\{\om_Y\})\le\lvol(\{\pi^\star\theta\})=\lvol(\{\theta\}), 
$$
where the last equality holds by Proposition~\ref{prop:lvol}~(iv). 

 Conversely, pick $\e>0$. By the first part of the proof, we can find $0\le T\in\{\theta\}$ with analytic singularities such that $\vol(\{\theta\})<\lvol(T)+\e$ and $T\geq \omega_X$, after rescaling $\omega_X$. Pick a modification $\pi\colon Y\to X$ such that $\pi^\star T$ has divisorial singularities, \ie $\pi^\star T=\b+[E]$ with $\b$ smooth and $E$ an effective $\Q$-divisor. Then $T\ge\om_X$ implies $\pi^\star T\ge\pi^\star\om_X$, and hence $\b\ge\pi^\star\om_X$. 

After passing to a higher modification, we may assume that $\pi$ is obtained as a finite sequence of blowups with smooth centers, so that there exists a $\pi$-exceptional $\Q$-divisor $D\ge 0$ on $Y$ such that $-D$ is $\pi$-ample. We can then find a smooth closed $(1,1)$-form $\theta_D\in\{D\}$ such that $\pi^\star\om_X-c\theta_D>0$ for all $c>0$ small enough. Then 
$$
\{\pi^\star\theta\}=\{\pi^\star T\}=\{\om_c\}+\{E+cD\},
$$
with $\om_c:=\b-c\theta_D\ge\pi^\star\om_X-c\theta_D>0$ for $0<c\ll 1$. By Lemma~\ref{lem:volmin}, we further have $\lvol(\om_c)=\lvol(\{\om_c\})$. By Proposition~\ref{prop:volbig}, this converges to $\lvol(\{\b\})=\lvol(\b)=\lvol(T)$ as $c\to 0$, and hence $\lvol(\{\theta\})<\lvol(\om_c)+\e$ for $c$ small enough. The result for $\lvol$ follows. The same argument applies for the upper volume $\uvol$, finishing the proof.  
\end{proof}

%
\subsubsection{General $\ddc$-classes} 
Assume now that $\{\theta\}\in\BC^{1,1}(X,\R)$ is psef. Then $\{\theta+\e\om_X\}$ is big for each $\e>0$, and its lower and upper volumes are both monotone increasing functions of $\e$ (see Proposition~\ref{prop:volbig}). We may thus introduce: 

\begin{defi}\label{defi:volpsef} 
The \emph{lower volume} and \emph{upper volume} of an arbitrary class $\{\theta\}\in\BC^{1,1}(X)$ are respectively defined by setting
\begin{equation}\label{equ:volpsef}
\lvol(\{\theta\}):=\inf_{\e>0}\lvol(\{\theta+\e\om_X\})=\lim_{\e\to 0}\lvol(\{\theta+\e\om_X\}),
\end{equation}
\begin{equation}\label{equ:uvolpsef}
\uvol(\{\theta\}):=\inf_{\e>0}\uvol(\{\theta+\e\om_X\})=\lim_{\e\to 0}\uvol(\{\theta+\e\om_X\})
\end{equation}
if $\{\theta\}$ is psef, and $\lvol(\{\theta\})=\uvol(\{\theta\}):=0$ otherwise. 
\end{defi}
Proposition~\ref{prop:volbig} implies that the definition is independent of the choice of $\om_X$, 
and is compatible with the big case. 

\begin{prop}\label{prop:volmon} Pick $\{\theta\}\in\BC^{1,1}(X)$. Then:
\begin{itemize}
\item[(i)] for any modification $\pi\colon Y\to X$, we have 
$$
\lvol(\{\pi^\star\theta\})=\lvol(\{\theta\}),\quad\uvol(\{\pi^\star\theta\})=\uvol(\{\theta\});
$$
\item[(ii)] for any $\{\theta'\}\in\BC^{1,1}(X)$ such that $\{\theta'\}-\{\theta\}$ is psef, we have 
$$
\lvol(\{\theta'\}\ge\lvol(\{\theta\}),\quad\uvol(\{\theta'\}\ge\uvol(\{\theta\});
$$
\item[(iii)] if $\theta\ge 0$ then $\lvol(\{\theta\})=\lvol(\theta)$. 
\end{itemize}
\end{prop}

\begin{proof} 
To prove (i), we may replace $\theta$ with $\theta+\e\om_X$ and assume that $\{\theta\}$ and $\{\pi^\star\theta\}$ are big. The result is then a direct consequence of Proposition~\ref{prop:lvol}~(iv). 

To prove (ii), we may assume that $\{\theta\}$ is psef, as the result is otherwise trivial. By~\eqref{equ:volpsef}, we may further replace $\theta,\theta'$ with $\theta+2\e\om_X,\theta'+\e\om_X$ with $\e>0$, and hence assume that $\{\theta\},\{\theta'\}$ and $\{\theta'-\theta\}$ are big. Pick a strictly positive current $S\in\{\theta'-\theta\}$ with analytic singularities. For each $0\le T\in\{\theta\}$ with analytic singularities, $0\le T+S\in\{\theta'\}$ has analytic singularities as well. By Proposition~\ref{prop:vmono}, we infer $\lvol(\{\theta'\})\ge\lvol(T+S)\ge\lvol(T)$, and taking the supremum over $T$ yields $\lvol(\{\theta'\})\ge\lvol(\{\theta\})$. The same argument applies to $\uvol$, proving (ii). 

Next finally $\theta\ge 0$. When $\{\theta\}$ is big, the equality $\lvol(\{\theta\})=\lvol(\theta)$ follows from 
Lemma~\ref{lem:volmin}. In the general case, we thus have $\lvol(\{\theta+\e\om_X\})=\lvol(\theta+\e\om_X)$ for $\e>0$. Now $\lvol(\theta)\le \lvol(\theta+\e\om_X)\le \lvol(\theta)+O(\e)$ by Proposition~\ref{prop:vmono} and~\eqref{equ:vsum}, hence $\lvol(\theta+\e\om_X)\to \lvol(\theta)$, and (iii) follows.
\end{proof}

\begin{rmk} When $\{\theta\}$ is further assumed to be big, the analogue of (iii) also holds for $\uvol$, by Lemma~\ref{lem:volmin}. However, we do not know whether this extends to the psef case, as $\lim_{\e\to 0}\uvol(\theta+\e\om_X)=\uvol(\theta)$ is then unclear. 
\end{rmk}

\begin{prop}\label{prop:volBC} The lower and upper volume functions coincide on the subspace 
$$
H^{1,1}_\BC(X,\R)\subset\BC^{1,1}(X,\R)
$$
of usual (closed) Bott--Chern classes. If $\{\theta\}\in H^{1,1}_\BC(X,\R)$ is further nef, then 
$$
\lvol(\{\theta\})=\uvol(\{\theta\})=\int_X\theta^n.
$$
\end{prop}
We simply denote by 
$$
\vol\colon H^{1,1}_\BC(X,\R)\to\R_{\ge 0}
$$
the common restriction of $\lvol,\uvol$.

\begin{proof} Assume $\theta$ is closed. If $\{\theta\}$ is not psef, then $\lvol(\{\theta\})=\uvol(\{\theta\})=0$ by definition. Assume that $\{\theta\}$ is psef, and pick $0\le T_\e\in\{\theta+\e\om_X\}$ with minimal singularities for each $\e>0$. Since $\theta$ is closed, we have $\D_\theta=0$ (see~\eqref{equ:Dthe}), and hence $\D_{\theta+\e\om_X}=O(\e)$. By Theorem~\ref{thm:monomass}, we infer $\uvol(T_\e)\le\lvol(T_\e)+O(\e)$, \ie 
$$
\uvol(\{\theta+\e\om_X\})\le\lvol(\{\theta+\e\om_X\})+O(\e),
$$
by Lemma~\ref{lem:volmin}. Letting $\e\to 0$ proves $\lvol(\{\theta\})=\uvol(\{\theta\})$. 

Assume next $\{\theta\}$ is further nef. Rescaling $\omega_X$ we can assume $\theta\le \omega_X$. For each smooth function $\f\in\PSH(\theta+\e\omega_X)\subset\PSH(X,2\om_X)$, the bounded mass property and Stokes' theorem yield
\begin{flalign*}
	\int_X (\theta+\e \omega_X+dd^c \f)^n &= \int_X (\theta+dd^c \f)^n + \sum_{k=1}^n \binom{n}{k} \e^k \int_X (\theta+dd^c\f)^{n-k} \wedge \omega_X^k\\
	& = \int_X \theta^n +   \sum_{k=1}^n \binom{n}{k} \e^{k} \int_X(\theta-2\omega_X +2\omega_X
	+dd^c\f)^{n-k} \wedge \omega_X^k\\
	& = \int_X \theta^n +O(\e). 
\end{flalign*}
Taking the infimum over all $\f$ yields 
$
\lvol(\{\theta+\e\om_X\})=\int_X\theta^n+O(\e),
$
hence $\lvol(\{\theta\})=\int_X\theta^n$.
\end{proof}

%
\subsubsection{The positive volume property}\label{sec:posvol}

As a consequence of Proposition~\ref{prop:volmon} we have: 

\begin{cor}\label{cor:posvol} 
We either have $\lvol\equiv 0$ on $\BC^{1,1}(X,\R)$, or $\lvol>0$ on the big cone.
\end{cor}
\begin{proof} Assume there exists $\{\theta_0\}\in\BC^{1,1}(X,\R)$ with $\lvol(\{\theta_0\})>0$. For any big class $\{\theta\}$, $t\{\theta\}-\{\theta_0\}$ is big for $t\gg 1$, and hence $t^n\lvol(\{\theta\})\ge\lvol(\{\theta_0\})>0$, by Proposition~\ref{prop:volmon}. 
\end{proof}

\begin{defi} We say that $X$ has the \emph{positive volume property} if $\lvol>0$ on the big cone of $\BC^{1,1}(X,\R)$. 
\end{defi}
By Proposition~\ref{prop:volmon} and Corollary~\ref{cor:posvol}, the positive volume property holds iff $\lvol(\om)>0$ for some Hermitian form $\om>0$ on $X$. 

Like the bounded mass property, the positive volume property is bimeromorphically invariant (by Proposition~\ref{prop:volmon}, or~\cite[Theorem~3.7]{GL22}). It always holds in dimension $n=2$;
when $n=3$ it holds for Vaisman manifolds \cite[Proposition 3.10]{AGL23}, as well as for some nilmanifolds \cite[Theorem 4.4]{AGL23}
and solvmanifolds \cite[Theorem 4.7]{AGL23}.
No counterexample is known at the moment of this writing. 

\smallskip

Recall that the bounded mass property is inherited by any submanifold. Here we   
establish a `dual' result for the positive volume property:

\begin{prop} 
Assume $\pi\colon X\to Y$ is a holomorphic submersion onto a compact complex manifold. 
If $X$ has the positive volume property, then so does $Y$.
\end{prop}

\begin{proof} Set $d:=\dim Y$, and pick a hermitian form $\beta_X$ on $X$. We claim that there exists a semipositive $(1,1)$-form $\theta\ge 0$ on $X$ that is (strictly) positive along the fibers of $\pi$, and satisfies $\theta^{n-d+1}=0$. Indeed, pick a $\mathcal C^\infty$-projection $p\colon T_X\to T_{X/Y}$ onto the holomorphic subbundle $T_{X/Y}=\ker d\pi\hto T_X$. The claim then holds with 
 $\theta(v,w):=\beta_X(p(v),p(w))$.  Given a Hermitian form $\om_Y>0$ on $Y$, we then have $\om_X:=\pi^\star\om_Y+\theta>0$. For any smooth $\f\in\PSH(\om_Y)$, we further have 
$$
(\om_X+\ddc\pi^\star\f)^n=(\pi^\star(\om_Y+\ddc\f)+\theta)^n={n\choose d}\pi^\star(\om_Y+\ddc\f)^d\wedge\theta^{n-d},
$$
and hence 
$$
0<\lvol(\om_X)\le{n\choose d}\int_Yf (\om_Y+\ddc\f)^d\le C\int_Y(\om_Y+\ddc\f)^d, 
$$
with  $0<f:=\pi_\star\theta^{n-d}\in C^\infty(Y)$ (see for instance~\cite[Section 2.15]{DemBook}).
We infer $\lvol(\om_Y))>0$, and the result follows. 
\end{proof}

 %
\subsection{Characterization of the Fujiki class}  \label{sec:fujiki}

Generalizing~\cite[Corollary~7.11]{B02} (which dealt with closed classes on a K\"ahler manifold), we next show: 

\begin{thm}\label{thm:bigvol} The lower volume function $\lvol\colon\BC^{1,1}(X,\R)\to\R_{\ge 0}$ is continuous. In particular, it vanishes outside of the big cone. 
\end{thm}

In other words, the last point states that $\lvol(\{\theta\})>0$ implies $\{\theta\}$ is big. This extends~\cite[Theorem~4.6]{GL22}, which dealt with nef classes. Note that the analogue of Theorem~\ref{thm:bigvol} fails for the upper volume:

\begin{exam} By~\cite[Example~3.5]{GL22}, any product $X=C\times C'$ of two curves carries a semipositive $(1,1)$-form $\theta$ such that $\int_X\theta^2>\lvol(\{\theta\})=0$. Since $X$ has the positive volume property (in fact $X$ is K\"ahler, hence $\lvol(\omega_X)>0$) $\{\theta\}$ is therefore not big, while $\uvol(\{\theta\})\ge\int_X\theta^2>0$. 
\end{exam}

\begin{proof}[Proof of Theorem~\ref{thm:bigvol}] The second point formally follows from the first one, since any $\{\theta\}$ on the boundary of the psef cone satisfies $\lvol(\{\theta-\e\om_X\})=0$ for $\e>0$ (see Definition~\ref{defi:volpsef}. However, we actually first show the former, by adapting (as in~\cite[Theorem~4.6]{GL22}) a strategy due to Chiose~\cite[Theorem~3.1]{Chi}. 

Assume thus that $\{\theta\}\in\BC^{1,1}(X,\R)$ satisfies $\lvol(\{\theta\})>0$. By Definition~\ref{defi:volpsef}, $\{\theta\}$ is necessarily psef, and we need to show that it is big, \ie that $\{\theta-\d\om_X\}$ is psef for some $\d>0$. By Hahn--Banach (see Lemma~\ref{lem:HB}), this is equivalent to showing $\int_X\theta\wedge\g\ge\d$ for each $\ddc$-closed $(n-1,n-1)$-form $\g>0$, normalized so that the volume form 
$$
\mu:=\om_X\wedge\g
$$
has mass $1$. Applying Proposition~\ref{prop:Fuj} to the big classes $\{\theta+\e\om_X\}$, $\e>0$, yields a sequence of modifications $\pi_j\colon Y_j\to X$ and decompositions 
$$
\pi_j^\star\{\theta+\e_j\om_X\}=\{\om_j\}+\{E_j\}
$$
where $\om_j$ a Hermitian form on $Y_j$, $E_j$ an effective $\Q$-divisor, $\e_j\to 0$ and $\lvol(\om_j)\to\lvol(\{\theta\})$. 

By~\cite[Theorem 4.1]{GL23}, for each $j$ there exists a positive current $0\le S_j\in\{\om_j\}$ with bounded potentials, smooth outside the exceptional locus of $\pi_j$, such that  
$$
S_j^n=c_j \pi_j^*\mu
$$
for some constant $c_j>0$. Note that $c_j=\int_X S_j^n\ge\lvol(\om_j)$, 
and hence 
\begin{equation}\label{equ:cj}
\liminf_j c_j\ge\lvol(\{\theta\}).
\end{equation} 
Denote by $T_j\in\{\theta+\e_j\om_X\}$ the unique positive current such that $\pi_j^\star T_j=S_j+[E_j]$. Then $T_j$ is smooth on a nonempty Zariski open set $\Om_j\subset X$, and satisfies $T_j^n=c_j\mu$. 

By Lemma~\ref{lem:bdmass2} we have $\int_XT_j^{n-1}\wedge\om_X\le M$ for a constant $M>0$ only depending on $\{\theta\}$ and $\{\om_X\}$. Introduce the set 
$$
A_j:=\left\{ x \in \Omega_j\mid T_j^{n-1}\wedge\omega_X\ge  2M \mu\text{ at }x\right\},
$$
and note that 
$$
\mu(A_j)\le\frac{1}{2M}\int_{A_j} T_j^{n-1}\wedge\om_X\le\frac 12, 
$$
and hence $\mu(\Om_j\setminus A_j)\ge\tfrac 12$. 
At each point of $\Omega_j \setminus A_j$ we have 
$$
T_j^{n-1}\wedge \omega_X\le 2M \mu=\frac{2M}{c_j}T_j^n, 
$$
which implies $T_j\ge\frac{c_j}{2nM}\om_X$, and hence $T_j\wedge\g\ge\frac{c_j}{2nM}\mu$. Since $\g$ is $\ddc$-closed, we infer
$$
\int_X(\theta+\e_j\om_X)\wedge\g=\int_X T_j\wedge\g\ge\frac{c_j}{2n M}\mu(\Om_j\setminus A_j)\ge\frac{c_j}{4nM}, 
$$
and~\eqref{equ:cj} yields the desired estimate 
$$
\int_X\theta\wedge\g\ge\d:=\frac{\lvol(\{\theta\})}{4nM}>0. 
$$

\smallskip

We now show that the lower volume function is continuous on $\BC^{1,1}(X,\R)$. By Proposition~\ref{prop:volbig}, it is continuous on the big cone, and it is identically zero outside the psef cone. Assume $\{\theta\}\in\BC^{1,1}(X,\R)$ is psef but not big. By the first part of the proof, $\lvol(\{\theta\})=0$, and we thus need to show $\lvol(\{\theta_j\})\to 0$ for any smoothly convergent sequence $\theta_j\to\theta$. We have $\theta_j\le\theta+\e_j\om_X$ with $\e_j\to 0$, and hence $\lvol(\{\theta_j\})\le\lvol(\{\theta+\e_j\om_X\})$, by Proposition~\ref{prop:volmon}. By definition, $\lvol(\{\theta+\e_j\om_X\})$ converges to $\lvol(\{\theta\})=0$, and we are done. 
\end{proof}

\begin{rmk} The above proof yields, more precisely, that $\{\theta\}$ is `more psef' than $\d\{\om_X\}$ for
$\d:=\lvol(\{\theta\})/4nM$ where 
$$
M:=\sup\left\{\int_X T^{n-1}\wedge\om_X\mid 0\le T\in\{\theta\}\right\}.
$$
\end{rmk}
 
 A simple consequence is the following interesting characterization of the Fujiki class, which establishes
 a transcendental version of the Grauert--Riemenschneider conjecture:
 
 \begin{cor}\label{cor:carFuj} 
Let $X$ be a compact complex manifold $X$ of dimension $n$.
The following properties are equivalent:
\begin{itemize}
\item[(i)] $X$ is a Fujiki manifold, \ie bimeromorphic to a compact K\"ahler manifold;
\item[(ii)] $X$ has the bounded mass property and it admits a closed positive $(1,1)$-current $T$ such that $\int_X T^n>0$.
\end{itemize}
\end{cor}

\begin{proof}
The implication (i)$\Rightarrow$ (ii) follows from Proposition \ref{pro:bmp}.
The existence of a closed positive $(1,1)$-current $T=\theta+dd^c \f$ such that $\int_X T^n>0$
is equivalent to $\vol(\{\theta\})>0$, by definition and Example \ref{exam:vclosed}.
By Theorem~\ref{thm:bigvol}, this implies that $\{\theta\}$ is big, and it then well-known (see for instance Proposition~\ref{prop:Fuj}) that $X$ admits a K\"ahler modification, and hence is Fujiki. 
\end{proof}

%
\section{Monge-Amp\`ere equations in big classes} \label{sec:solvMA}
%
In this section, we no longer asssume that $X$ has the bounded mass property. We pick a big class $\{\theta\}\in\BC^{1,1}(X,\R)$, fix a $\theta$-psh function 
 $\rho\le 0$ with analytic singularities such that $\theta+dd^c \rho \geq\omega_X$ (after rescaling $\om_X$), and denote by $\Omega\subset X$ the Zariski open set where $\rho$ is smooth.  We also fix a (smooth, positive) volume form $dV$ on $X$, and denote by 
 $$
 \|f\|_p:=\left(\int_X|f|^p dV\right)^{1/p}
 $$
 the $L^p$-norm of a measurable function $f$. We are going to solve complex Monge--Amp\`ere equations in $\Omega$,
 extending \cite[Theorem B]{BEGZ10} to the Hermitian setting.
 %
 \subsection{The domination principle}


We start by establishing a useful local maximum principle.

\begin{lem}\label{lem: CP via envelopes}
  Let $U\Subset \Omega$ be an open subset that is biholomorphic to a ball in $\mathbb C^n$. Let $u,v$ be bounded $\theta$-psh functions in $U$ such that $\liminf_{z\to \partial U}(u-v)(z) \geq 0$. If $(\theta+dd^c u)^n\leq (\theta+dd^c v)^n$ in $\{u<v\}$, then $u\geq v$ in $U$. 
\end{lem}

\begin{proof}
    Fix $\e>0$ and set $v_{\e}=\max\{u,v-\e\}$.  Then $u=v_{\e}$ near the boundary of $U$, and the Bedford--Taylor maximum principle (compare Proposition~\ref{prop:maxprinciple}) thus yields
     $$
    (\theta+dd^c v_{\e})^n \geq\one_{\{v\ge u+\e\}}(\theta+dd^c v)^n+\one_{\{u>v-\e\}}(\theta+dd^c u)^n
    $$
    on $U$. Since $(\theta+dd^c v)^n\ge (\theta+\ddc u)^n$ on $\{v\ge u+\e\}\subset\{v>u\}$, we infer
    $$
    (\theta+dd^c u)^n\le (\theta+dd^c v)^n. 
    $$ 
Let $\varphi:=\env^\star(u-v_{\e})$ denote the largest psh function on $U$ lying below $u-v_{\e}$ outside a pluripolar set, and 
    let $h$ be a defining psh function for $U$: $h=0$ on $\partial U$ and $h<0$ in $U$. 
    Since $u-v_{\e}\leq 0$ with equality near $\partial U$, we can find a large constant $A>0$ such that 
    $Ah\le u-v_{\e}$. Thus $Ah \leq \f$, hence $\varphi$ vanishes on $\partial U$. Since $\varphi+ v_{\e}\leq u$, and the two functions are $\theta$-psh, by the maximum principle we have 
    \[ 
   {\bf 1}_D (\theta + dd^c \varphi + dd^c v_{\e})^n  \leq {\bf 1}_D (\theta+dd^c u)^n,
    \]
    where $D=\{\varphi+ v_{\e} = u\}$ is the contact set. 
    Using 
    \[
    (\theta +dd^c u)^n \leq (\theta+dd^c v_{\e})^n \leq  (\theta + dd^c \varphi + dd^c v_{\e})^n,
    \] 
    we obtain 
    ${\bf 1}_D (dd^c \varphi)^n=0$. 
   Since $(dd^c \varphi)^n$ is supported on $D$, we infer that $(dd^c \varphi)^n=0$, and the 
    Bedford--Taylor comparison principle ensures that $\varphi=0$. 
    Thus $u\geq v_\e=\max\{u,v-\e\}$ outside a pluripolar set, and hence everywhere, and letting $\e\to 0$ we conclude that $u\geq v$. 
\end{proof}

\begin{lem}
	Let $u,v$ be $\theta$-psh functions on $X$ such that $\rho \leq \min\{u,v\}$ and 
	 $v\leq u+C$ for some constant $C$.
		If 
	\[
	{\bf 1}_{\{u<v\}}(\theta+dd^c u)^n   \leq c {\bf 1}_{\{u<v\}}(\theta+dd^c v)^n, 
	\]
	 for some $c\in [0,1)$, then $u\geq v$. 
\end{lem}


This domination principle contains a great deal of information.
While the statement is similar to \cite[Proposition 1.12]{GL23}, the proof 
we provide here is very different.

\begin{proof}
	Fix $\e\in (0,1)$ such that $(1-\e)^n>c$ and set $v_{\e}:= (1-\e)v + \e \rho$. Since $\rho$ satisfies $\theta +dd^c \rho \geq \omega_X$, we have 
$$\theta +dd^c \lambda \rho = \lambda (\theta+dd^c \rho) + (1-\lambda)\theta \geq \omega_X +(1-\lambda) \theta.$$
Thus, when $|\lambda-1|$ is sufficiently small we have $\theta +dd^c \lambda \rho \geq \omega_X/2$.  Replacing $\rho$ with $\lambda \rho -1$ for some $\lambda>1$ we can assume 
	\[
	\lim_{x\to \partial \Omega} (u-\rho)=+\infty\; \text{and}\; \rho -v \leq -1.
	\]
We are going to show that $u\geq v_{\e}$; the result then follows by letting $\e\to 0$. 
Assume by contradiction that 
	\[
	\inf_{\Omega} (u -v_{\e}) = m<0. 
	\]
	Let $(x_j)$ be a sequence in $\Omega$ converging to $x_0\in \Omega$ such that $\lim_j (u-v_{\e})(x_j)=m$. Take a holomorphic coordinate chart $(B,z)$ around $x_0$ so that $B=\{|z|<1\}\Subset \Omega$, and define 
	$
	\eta =-a|z|^2+a,
	$
	where the constant $a>0$ is so small that $\frac{\om_X}{2}+dd^c \eta\geq 0$. Setting
	\[
	\varphi := v_{\e} + \e \eta + m,
	\]
	we get  $\liminf_{x\to \partial B}(u-\varphi) \geq 0$, 
	\[
	\theta+dd^c \varphi = (1-\e)(\theta+dd^c v) + \e (\theta+dd^c\rho + dd^c \eta)\ge 0
	\] 
	and 
	$$
	(\theta+dd^c u)^n \leq c (\theta+dd^c v)^n\leq  (1-\e)^n (\theta+dd^c v)^n\leq  (\theta+dd^c\f)^n\; \text{on}\; \{u<\varphi\}
	$$
	because 
	\[
	\{u<\varphi\} \subset \{u< v + \e (\rho-v+\eta)+ m\} \subset \{u< v+m\}\subset \{u< v\},
	\]
	if $a$ is small enough.  Lemma \ref{lem: CP via envelopes} yields $u\geq \varphi$, and evaluating at $x_j$ we obtain 
	\[
	u(x_j) - v_{\e}(x_j) \geq  \e \eta (x_j) + m,
	\]
	which yields a contradiction as $j\to +\infty$, since $\eta(x_0)>0$. 
\end{proof}

As a simple consequence, we get: 
\begin{cor}\label{cor:MAcst}
	Assume $u,v \in \PSH(\theta)$ are such that $u-v$ is bounded, $\min\{u,v\}\geq \rho$, and 
	$$(\theta+dd^c u)^n \leq c(\theta+dd^c v)^n,$$
for some $c\geq 0$. Then $c\geq 1$. 
\end{cor}

%

\subsection{Solving complex Monge--Amp\`ere equations}

A crucial step in solving complex Monge--Amp\`ere equations is to establish uniform a priori estimates. 
Following \cite{GL23} we construct suitable subsolutions to bound the solutions via the domination principle.

\begin{lem}\label{lem: subsolution}
	Let $\eta \in \PSH(X,\theta)$ such that $0\geq \eta \geq \rho$. Let $g\geq 0$ be a measurable function on $X$ such that $\|g\|_q\leq 1$, for some exponent $q>1$. Then there exists $v\in \PSH(X,\theta)$ with $\eta\leq v \leq \eta+1$ such that
	\[
	(\theta+dd^c v)^n \geq m gdV
	\]
	on $\Omega$, where $m$ is a positive constant only depending on $q$, $(X,\omega_X)$, and $\theta$. 
\end{lem}

\begin{proof} Set $h:=g (2-\rho)^{2n}$. By H\"older's inequality, we have $\|h\|_r\leq C$ for some $r\in (1,q)$ and a uniform constant $C>0$, and~\cite[Lemma 3.3]{GL23} thus yields $u\in \PSH(X,\omega_X)$ such that $-1\leq u\leq 0$, and 
	\[
	(\omega_X+dd^c u)^n \geq  m hdV
	\]
	with $m>0$ uniformly bounded away from $0$.  
	Set $u_1:= \rho+u$. Since we assume $\theta+dd^c\rho\ge\om_X$, we get $\theta+dd^c u_1\ge\om_X+dd^c u\ge 0$, thus $u_1$ is $\theta$-psh and satisfies $(\theta+dd^c u_1)^n \geq m hdV$. 
	
	\smallskip
	
	Consider now $v := \eta+\frac{1}{1+\eta-u_1}$. Observe that 
	$1 \leq 1+\eta-u_1=1-u+\eta-\rho  \leq 2-\rho$ and
	\begin{eqnarray*}
	\theta+dd^c v
	&=& \theta+dd^c \eta -\frac{dd^c (\eta-u_1)}{(1+\eta-u_1)^2}+\frac{2 d(\eta-u_1) \wedge d^c(\eta-u_1)}{(1+\eta-u_1)^3} \\
	&\geq & \left[1- \frac{1}{(1+\eta-u_1)^2}\right](\theta+dd^c \eta) +\frac{\theta+dd^c u_1}{(1+\eta-u_1)^2} \\
	& \geq & \frac{\theta+dd^c u_1}{(2-\rho)^2}.
	\end{eqnarray*}
	The function  $v$ is thus $\theta$-psh, with $\eta\leq v\leq \eta+1$ and 
	\[
	(\theta+dd^c v)^n \geq \frac{(\theta+dd^c u_1)^n}{(2-\rho)^{2n}} = m gdV. 
	\]
\end{proof}

We now show that one can solve  Monge-Amp\`ere equations with respect to a  big form $\theta$. We first start with the equation twisted by an exponential. 

\begin{thm}\label{thm:MAAY}
Assume $0 \leq f \in L^p(X)$ with $p>1$ and $\|f\|_p>0$. 
For each $\lambda >0$, we can then find a unique 
$\f_{\lambda} \in PSH(X,\theta)$ such that $\f_\la=V_\theta+O(1)$ and 
$$
\left(\theta+dd^c \f_{\lambda}\right)^n=e^{\lambda \f_{\lambda}} fdV
$$
on $\Om$. Furthermore, $\sup_X|\f_\la-V_\theta|\le C$ for a constant $C>0$ only depending on $(X,\omega_X)$, $\theta$, $p$, $\|f\|_p$ and $\lambda$. 
\end{thm}

\begin{proof} Uniqueness follows from the domination principle. To simplify the notation, we assume that $\lambda=1$. 
To prove the existence,
	we approximate $V_{\theta}$ by a quasi-decreasing sequence $(\psi_j)$ of $\theta$-psh functions with analytic singularities:
 we first apply Demailly's approximation (Theorem \ref{thm:Demreg})
  to get a decreasing sequence $\tilde{\psi}_j \in \PSH(X,\theta+2^{-j}\omega)$ converging to $V_{\theta}$,  and we then take 
	\[
	\psi_j = (1-2^{-j})\tilde{\psi}_j + 2^{-j}\rho.
	\] 
	Observe that 
	$
	\psi_j \geq (1-2^{-j})V_{\theta} + 2^{-j} \rho \geq \rho,
	$
	thus 
	\[
	\env_\theta^\star\left(\inf_{k\geq j} \psi_k\right) \geq  (1-2^{-j})V_{\theta} + 2^{-j} \rho \nearrow V_{\theta} \; \text{as}\; j \to +\infty. 
	\]
	
	Resolving the singularities of $\p_j$ and arguing as in the proof of Theorem~\ref{thm:bigvol}, one deduces from~\cite[Theorem 3.4]{GL23} the existence of 
	$u_j \in \PSH(X,\theta)$ such that $u_j=\psi_j+O(1)$ and 
	\[
	(\theta+dd^c u_j)^n = e^{u_j} fdV. 
	\]
	By Lemma \ref{lem: subsolution} there exists $v_j\in \PSH(X,\theta)$ with $\psi_j-1\leq v_j \leq \psi_j\leq 0$ such that 
	\[
	(\theta+dd^c v_j)^n \geq m fdV \geq  e^{v_j+ \log m} f dV.
	\]
	The domination principle thus yields a uniform constant $C'$ such that
	\[
	\psi_j-C' \leq v_j-C'+1 \leq u_j.
	\] 
	Extracting a subsequence, we can assume that $u_j \to u \in \PSH(X,\theta)$ almost everywhere on $X$.  
	The above estimates yield  uniform bounds $V_{\theta}-C' \leq u \leq V_{\theta}$.

	We next consider the  sequences 
	\[
	\varphi_j = \env_\theta^\star\left(\inf_{k\geq j} u_k\right)
	\; \; \text{ and } \; \;  \tau_j = \mathrm{sup}^\star_{k\geq j} u_k. 
	\]
	These functions are locally bounded in $\Omega$
	since 
	\[
	\min\{\varphi_j,\tau_j\} \geq (1-2^{-j})V_{\theta}+2^{-j}\rho -C' \geq \rho-C'.
	\] 
	Note that $\tau_j$ decreases pointwise to $u$, while $\f_j$ increases (pointwise outside a pluripolar set)
	to some function $\f \in \PSH(X,\theta)$.
		The maximum principle moreover yields, in $\Omega$,
	\[
	(\theta+dd^c \varphi_j)^n \leq e^{\varphi_j}  f dV,
	\]
	and 
	\[ 
	(\theta+dd^c \tau_j)^n \geq e^{\tau_j} f dV.
	\]
	Letting $j \to +\infty$, we obtain $\varphi\leq u$, $\f=u+O(1)$, and 
	\[
	(\theta+dd^c \varphi)^n \leq e^{\varphi}  f dV,
	\; \; \text{ while }\; \;  
	(\theta+dd^c u)^n \geq e^{ u}  f dV.
	\]
	The domination principle ensures $\varphi=u$, hence
	$
	(\theta+dd^c u)^n=e^{u}fdV.
	$

\end{proof}

\begin{thm}
Assume $0 \leq f \in L^p(X)$ with $p>1$ and $\|f\|_p>0$.  Then there exists $(\varphi,c)\in \PSH(X,\theta)\times (0,+\infty)$ 
such that $V_{\theta} -C \leq \varphi \leq V_{\theta}$ and 
	$$
	(\theta+dd^c \varphi)^n = c fdV
 $$
	on $\Om$, where
	\begin{itemize}
	\item the constant $c>0$ is uniquely determined by $f,X,\theta$, and
	\item $C>0$ is a uniform constant that only depends on $(X,\omega_X)$, $\theta$, $p$ and $\|f\|_p$. 
	\end{itemize}  
\end{thm}

When $X$ is K\"ahler and $\theta$ is closed, this is \cite[Theorem B]{BEGZ10}.
When  $\theta>0$ is a Hermitian form, this result is a combination of 
\cite{TW10}
and 
\cite{DK12,KN15}.
When the class $\{\theta\} \in BC^{1,1}(X)$ is nef, one can further show the existence of a solution which is smooth
in $\Omega$ (see \cite[Theorem B]{GL23} and \cite[Theorem C]{Dang24}). 
The uniqueness of $\f$ up to an additive constant is largely open 
(see however \cite[Theorem A]{KN19}).

\begin{proof}
The uniqueness of the constant $c$ follows from the domination principle (see Corollary~\ref{cor:MAcst}). 
For each $j\in \mathbb N^*$ we use Theorem \ref{thm:MAAY} to find
$u_j \in \PSH(X,\theta)$ such that $u_j -V_{\theta}$ is bounded  and 
\[
(\theta+dd^c u_j)^n = e^{j^{-1} u_j} fdV.
\] 
Setting $v_j: = u_j-\sup_X u_j$ and $f_j: = e^{j^{-1} v_j} f$,
the above equation becomes 
\[
(\theta+dd^c v_j)^n = c_j f_jdV, 
\; \text{ where } \;
 c_j = e^{j^{-1}\sup_X u_j}.
\] 
Extracting and relabelling we can assume that $v_j \to v\in \PSH(X,\theta)$ in $L^1$ and
almost everywhere, hence $e^{j^{-1}v_j}  \to 1$ in $L^1(X,dV)$. 

Pick a Gauduchon metric $\omega_G$ on $X$ and write $\omega_G^n = hdV$. The mixed Monge-Amp\`ere inequality gives 
\[
(\theta+dd^c u_j) \wedge \omega_G^{n-1} \geq e^{j^{-1}v_j/n} c_j^{1/n} f^{1/n}h^{1-1/n} dV.  
\]
Integrating this inequality over $X$ yields an upper bound for $c_j$.

	Fix a small constant $\e>0$ such that $e^{-\e v_j} f$ is uniformly in $L^q(X,dV)$ for some $1<q<p$. This can be done thanks to the uniform version of Skoda's integrability theorem (see \cite{GZbook}, \cite{Zer01}). Using Lemma  \ref{lem: subsolution}, we can find $\eta_j \in \PSH(X,\theta)$ with $V_{\theta}-1\leq \eta_j\leq V_{\theta}$ such that 
	\[
	(\theta+dd^c \eta_j)^n \geq mfe^{-\e v_j}dV \geq mfe^{\e(\eta_j- v_j)}dV. 
	\]
	The domination principle yields $m\leq c_j$ and $v_j \geq \eta_j + \log m \geq V_{\theta}-C$, for a uniform constant $C>0$.  Extracting a subsequence, we can assume $c_j\to c>0$. 
	
		Consider the following sequences 
	\[
	\varphi_j = \env_\theta^\star\left(\inf_{k\geq j} v_k\right), \; \tau_j =\mathrm{sup}^\star_{k\geq j} v_k.
	\]
	Then $\min\{\varphi_j,\tau_j\} \geq (1-2^{-j})V_{\theta}+2^{-j}\rho \geq \rho$. In particular, these functions are locally bounded in $\Omega$. 
	By the maximum principle and the minimum principle, we have
	\[
	(\theta+dd^c \varphi_j)^n \leq e^{\e(\varphi_j- \inf_{k\geq j}v_k)} \sup_{k\geq j} (c_kf_k) dV,
	\]
	and 
	$$
	(\theta+dd^c \tau_j)^n \geq e^{\e(\tau_j- \sup_{k\geq j}v_k)} \inf_{k\geq j} (c_kf_k) dV,
	$$
	in $\Omega$. 
	Letting $j \to +\infty$, we obtain $V_{\theta} -C\leq \varphi\leq v\leq V_{\theta}$, while
	\[
	(\theta+dd^c \varphi)^n \leq e^{\e(\varphi-v)}  cf dV
	\; \; \text{ and }\; \;
	 (\theta+dd^c v)^n \geq e^{\e(v- v)}  cf dV.
	\] 
	The domination principle thus gives $\varphi=v$ and 
$(\theta+dd^c v)^n=cfdV$,
	finishing the proof.  
\end{proof}

\end{document}